\documentclass[a4paper, 12pt, oneside, reqno, notitlepage]{amsart}
\usepackage{amsmath,amssymb,amsthm,graphicx,mathrsfs,bbm,url}
\usepackage[margin=3cm]{geometry}
\usepackage{amsthm}
\usepackage{enumitem}
\usepackage{mathtools}
\usepackage[usenames,dvipsnames]{color}
\usepackage[colorlinks=true,linkcolor=teal,citecolor=cyan]{hyperref}
\usepackage[super]{nth}
\usepackage[open, openlevel=2, depth=3, atend]{bookmark}
\hypersetup{pdfstartview=XYZ}
\usepackage[font=footnotesize]{caption}
\usepackage[colorinlistoftodos]{todonotes}
\usepackage{mathabx}
\usepackage{tikz-cd}
\tikzcdset{row sep/normal=3.6em,column sep/normal=3.6em}

\usepackage{subcaption}

\RequirePackage[style=alphabetic,giveninits=true,maxnames=50]{biblatex} 
\usepackage{subcaption}

\usepackage{epstopdf}

\theoremstyle{plain}
\newtheorem{theorem}{Theorem}[section]
\newtheorem{lemma}[theorem]{Lemma}

\newtheorem{corollary}[theorem]{Corollary}

\theoremstyle{definition}
\newtheorem{definition}[theorem]{Definition}

\theoremstyle{remark}
\newtheorem{remark}[theorem]{Remark}

\newcommand{\dom}{\mathrm{dom}}

\title
[Spectral submanifolds for nonlinear PDEs] 
{Spectral submanifold reduction for PDEs describing nonlinear continuum vibrations}

\author{Gergely Buza}
\address{Institute for Mechanical Systems, ETH Zürich, Leonhardstrasse 21, 8092 Zurich, Switzerland}
\email{buzag@ethz.ch}

\author{George Haller}
\address{Institute for Mechanical Systems, ETH Zürich, Leonhardstrasse 21, 8092 Zurich, Switzerland}
\email{georgehaller@ethz.ch}

\begin{document}

\begin{abstract}
We prove the existence of spectral submanifolds in nonlinear partial differential equations (PDEs) describing forced-damped continuum vibrations. Our results cover structural vibrations subject to generalized structural damping. Based on these results, backbone curves and forced response curves can be rigorously extracted from PDEs, without a priori discretization. On two examples involving an elastic beam and a thin plate, we show that backbone and forced response curves can even be calculated by hand directly from the PDEs.
\end{abstract}

\maketitle


\section{Introduction}

Invariant manifolds serve a dual purpose in nonlinear vibration problems.
For one, they provide practically useful reduced-order models (ROMs) that suffer less from the dimensionality requirements characteristic of their popular projection-based counterparts borne out of linear theory.
Beyond this, invariant manifolds serve a physically more profound purpose: they isolate physical phenomena pertinent only to a select few vibration modes in such a way that keeps the nonlinear structure of the equations intact.
This isolation of modal dynamics in nonlinear systems is facilitated by \textit{spectral submanifolds} (SSMs), which are invariant manifolds tangent to modal subspaces at equilibria \cite{haller2025modeling}, and are the content of the present work.
SSMs have enabled classical notions from the vibrations literature, such as the backbone curve, to be given a more refined theoretical interpretation, which has helped identify the mechanism underlying the alteration of resonant frequencies caused by nonlinear effects \cite{breunung2018explicit}.

However, through the course of its development, the theory has had a notable gap: with the exception of \cite{kogelbauer2018rigorous}, it has remained generally unapplicable to nonlinear continuum vibrations described by partial differential equations (PDEs).
Such PDEs generate infinite-dimensional dynamical systems that are driven by unbounded operators on account of the spatial derivatives.
Due to the ensuing technical difficulties, SSM reduction has only been carried out with mathematical rigor in finite-dimensional discretizations of the underlying PDEs.
This leaves open the following question: Can the existence of SSM-reduced models be maintained in the continuum limit, i.e., do SSMs exist in the infinite-dimensional system describing the PDE?
Or, phrased with the application in mind: Can one rigorously reduce the PDEs of nonlinear structural dynamics to very low-dimensional ODEs defined on SSMs?
We answer this question positively in the present work for a large class of PDEs describing nonlinear continuum vibrations.

As ROMs, invariant manifolds have by now achieved a prominent status
in the engineering mechanics literature. 
Their permeation into the field is commonly attributed to the seminal work of Shaw and Pierre \cite{shaw1993normal} on nonlinear normal modes, who also extended their original approach, at least to the extent of formal calculations, to PDEs describing continuum systems \cite{shaw1994normal}.
In the subsequent two decades, the field evolved largely through practical considerations concerning nonlinear normal modes \cite{touze2006nonlinear,kerschen2009nonlinear,peeters2009nonlinear}; see also \cite{mikhlin2024nonlinear} for a recent review.
Alternative approaches have also emerged since then, including the quadratic manifold approach \cite{jain2017quadratic,rutzmoser2017generalization}.

A solid theoretical foundation to these approaches came through the work of Haller and Ponsioen \cite{haller2016nonlinear,haller2017exact} on SSMs, utilizing the theoretical results of \cite{Cabre2003a,cabre2003parameterization2,cabre2005parameterization3,haro2016parameterization} and \cite{sternberg1957local,sternberg1958structure}.
SSMs are defined as nonlinear analogues of spectral subspaces/subbundles, with the overarching objective of subdividing the dynamics according to asymptotic properties of trajectories in the vicinity of a stationary state or invariant set.
The original theory of SSMs has undergone significant extensions since then both in the practical and theoretical realms, motivated predominantly by contemporary problems of increasing physical complexity; for a recent review, see \cite{haller2025modeling}.
Most pertinent to the present setting are the early works on the von Kármán beam \cite{jain2018exact} and backbone curves \cite{breunung2018explicit}; the computational guidelines proposed for finite element models \cite{jain2022compute}; the treatment of internal resonances \cite{li2022nonlinear,li2022nonlinear2};
the beam-buckling example which motivated the introduction of mixed-mode SSMs \cite{haller2023nonlinear};  the recent treatment of nonsmooth problems \cite{morsy2025data}; and the streamlined process of selecting optimal initial conditions on SSMs via oblique projections \cite{bettini2025data} that are best suited to capture the backbone curve.
An automated algorithm (SSMTool) for computing SSMs for general, finite-element models was developed by \cite{jain2022ssmtool}. 
An alternative implementation for systems with linear damping and cubic nonlinearities was given by \cite{touze2021model}.
Another algorithm (SSMLearn) has been developed by \cite{cenedese2022data} to extract SSMs purely from data.

There exist only a few studies concerned with the rigorous extraction of SSMs directly from an infinite-dimensional dynamical system.
A continuum beam model was the subject of the first such approach \cite{kogelbauer2018rigorous}, incompressible fluid flows were considered in \cite{buza2024spectral} (with data-driven counterparts in \cite{kaszas2022dynamics,kaszas2024capturing}), and delay differential equations were considered in \cite{szaksz2025spectral,buza2025existence} (with data-driven counterparts in \cite{abbasciano2026data,abbasciano2026data2}).
In \cite{kogelbauer2018rigorous}, the authors successfully apply the parameterization method \cite{Cabre2003a} directly to the PDE describing the beam; however, in order to make the linearized operator invertible (as is necessary to apply \cite{Cabre2003a} to the slow subspace), they have to go to great lengths to constrain the spectrum to be bounded in real part, which moreover ends up at odds with experimentally observed damping ratios (see specifically \cite{banks1991damping}).
Here, we remedy this situation by resorting to a simpler invariant manifold theory originally due to Irwin \cite{irwin1980new}, which relies on weaker assumptions, but in turn loses the uniqueness conclusion of \cite{Cabre2003a} for a smoothest SSM. 
This is, however, of no concern in the increasingly important data-driven construction of SSMs where the most influential (rather than the smoothest) one is inferred from the available data.

An important practical objective of invariant manifolds in vibrations is to reduce the computational complexity required to obtain \textit{forced response curves} (FRCs).
FRCs provide the main guiding principle in the design and monitoring of mechanical structures subject to external periodic forcing.
FRCs depict some measure of response amplitude as a function of forcing frequency for a given forcing amplitude. 
These curves will generally develop peaks, which, in the limit of low amplitude forcings, occur at the natural frequencies of the linear system.
As the forcing amplitude is increased, the peaks shift away either to increasing frequencies (hardening behavior) or to decreasing frequencies (softening behavior) depending on the specific character of the nonlinearity.
The graphs that the evolution of these peaks trace in the forcing frequency -- response amplitude plane are usually termed \textit{backbone curves}.
As shown by \cite{breunung2018explicit} rigorously, these forced backbone curves can be closely approximated by the damped backbone curves, which are simply instantaneous amplitude-frequency plots of decaying vibrations within SSMs of the unforced limit of the mechanical system. We will use our present results to extract such backbone curves from nonlinear continuum vibrations directly.

\subsection{Class of partial differential equations considered}
We consider evolution equations of the form
\begin{equation}
    \ddot{u} + B\dot{u} + Au = f(u,\dot{u}) + \varepsilon h(t)
    \label{eq:second_order_intro}
\end{equation}
on a Hilbert space $X$ (most commonly the function space $L^2$).
Here $u$ is an element of the space $X$, which usually stands for the displacement vector field over the physical coordinates evolving in time.
The operators $A$ and $B$ stand for the elastic and damping operators, respectively, whereas $f$ and $h$ stand for the nonlinearity and external forcing.
We assume $B$ to be as such that the origin is stable according to the linear, unperturbed ($\varepsilon=0$) dynamics.
A more precise characterization of these terms will be given in Section~\ref{sect:setup}.

We assume generalized structural damping in the sense of \cite{chen1989proof}, who let $A^\alpha \lesssim B \lesssim A^\alpha$ for some $\alpha \in [\frac12,1]$ (see \eqref{eq:ABA_assumption}). 
As \cite{chen1989proof} shows, this assumption implies the analiticity of the semigroup generated by the linear part of \eqref{eq:second_order_intro}, a very desirable property for our purposes. 
This notion of damping is far more permissive than what structural damping classically stands for (which is $B = A^\frac12$, \cite{chen1982mathematical}), as illuminated in, e.g., Figure~\ref{fig:1}.
Many of the  damping models proposed throughout the years (e.g., \cite{russell1992mathematical,russell1993general}) require this level of generality.

The first example that we detail herein is an extended version of the beam model of \cite{formica2013coupling}, (originally due to \cite{mettler1962dynamic}),
\begin{displaymath}
              \frac{\partial^2 w}{\partial t^2} - 2 \mu \frac{\partial^3 w}{\partial x^2 \partial t}
+ \frac{\partial^4 w}{\partial x^4}
= \left(  \frac{a}{2 } \int_0^1 \left( \frac{\partial w}{\partial x} \right)^2 dx \right)
\frac{\partial^2 w}{\partial x^2} + \varepsilon h(t,x),
\end{displaymath}
with parameters $\mu,\varepsilon>0, \;a \in \mathbb{R}$, external forcing $h$  and vertical displacement $w$.
Our second example is a Kirchhoff-Love plate model,
\begin{displaymath}
    \ddot{w} - 2\mu \Delta \dot{w} + \Delta^2 w = \kappa w^3,
\end{displaymath}
posed on a rectangular domain, with $\kappa \in \mathbb{R}$.
We consider both models with simply supported boundary conditions.
While this makes the hand calculations performed in Sections~\ref{sect:beam}-\ref{sect:plate} significantly easier, our approach also extends to more general boundary conditions.
Several other beam models that fit our assumptions are explored in Section~\ref{sect:beam}.

\subsection{Contributions and organization of the paper}

The following list of points outlines the contributions of this paper.
\begin{enumerate}[label =(C.\arabic*)] 
    \item \label{C1} We show the existence of slow manifolds tangent to the spectral subspace corresponding to the least damped modes. 
    We do so by an application of \cite[Theorem~1.1]{chen1997invariant}, which also provides a complementary foliation that describes the synchronized descent of trajectories towards the slow manifold and the precise rate at which they converge to it.
    The theory is then improved to encompass periodically forced systems and permit arbitrary spectral subsets through a sequence of remarks.
    \item \label{C2} An application of the results of \ref{C1} overcomes the limitations -- with regards to damping models -- faced in \cite{kogelbauer2018rigorous}. 
    These results also bridge the gap in the theory of reducibility, which enables the transfer of our physical understanding outlined at end of the introduction from the discretized model to the PDE.
    \item \label{C3} We  generalize and make systematic the computational procedure outlined in \cite{kogelbauer2018rigorous} (based on the parameterization method \cite{haro2006parameterization}) by adapting the bi-orthogonal system described in \cite[Appendix~A]{chen1989proof}.
    We provide examples in which forced response and backbone curves are calculated analytically and rigorously directly from the PDE. 
\end{enumerate}

The rest of the paper is organized as follows.
The main theoretical body concerning point \ref{C1} is Section~\ref{sect:setup}.
This section is fairly technical, but only the introductory subsection (Section~\ref{sect:actual_setup}) describing the setup and the statement of Theorem~\ref{thm:main} (alongside its assumptions \ref{A1}-\ref{A4}) are necessary for the subsequent sections; the reader uninterested in the precise technical details is invited to skip the rest.
The systematic approach mentioned in \ref{C3} is detailed in Section~\ref{sect:comp}.
Examples for beam (Section~\ref{sect:beam}) and plate (Section~\ref{sect:plate}) models are described thereafter.
Conclusions are drawn in Section~\ref{sect:conclusions}.
Appendices \ref{appendix:domA12} and \ref{appendix:domLB} are technical ones that aid the identification of the exact function spaces at play in Sections~\ref{sect:beam}-\ref{sect:plate}; the latter also helps clarify assumption \ref{A3x}. 
The remaining two appendices, \ref{appendix:biorthogonal} and \ref{appendix:biorthogonal_plate}, merely collect some quantities used during the course of explicit calculations in Sections~\ref{sect:beam}-\ref{sect:plate}.

\section{Theoretical setup and results}
\label{sect:setup}

\subsection{Setup}
\label{sect:actual_setup}

To describe damped nonlinear continuum vibrations,
we adopt the mathematical setup of \cite{chen1989proof}.
For this, we make the following set of assumptions.

\begin{enumerate}[label =(A.\arabic*)] 
    \item \label{A1} \textbf{(Undamped system)} 
    Let $X$ denote a Hilbert space with inner product $\langle \cdot , \cdot \rangle_X$.
    Let $A : \dom(A) \to X$ be a \textit{strictly} positive self-adjoint operator with dense domain $\dom(A) \subset X$ and compact resolvent.
    
    \item \label{A2} \textbf{(Damping)}  Let $B : \dom(B) \to X$ be a   positive self-adjoint operator with dense domain $\dom(B) \subset X$.
    There are two constants $0 < \rho_1 < \rho_2 < \infty$ and $\alpha \in [\frac12,1]$ such that $\dom(B^\frac12) = \dom(A^\frac{\alpha}{2})$ and
    \begin{equation}
        \rho_1 \langle A^\alpha u,u \rangle_X \leq \langle B u,u \rangle_X \leq \rho_2 \langle A^\alpha u,u \rangle_X, \qquad u \in \dom(B^\frac12).
        \label{eq:ABA_assumption}
    \end{equation}
\end{enumerate}

Under assumptions \ref{A1}-\ref{A2}, we consider the abstract evolution equation
\begin{equation}
    \ddot{u} + B\dot{u} + Au = f(u,\dot{u}), \qquad \text{on } X. 
    \label{eq:second_order}
\end{equation}
Motivated by the fact that the undamped ($B = 0$), linear ($f =0$) part of \eqref{eq:second_order} preserves the energy (see \cite{chen1982mathematical})
\begin{displaymath}
      E(u,\dot{u}) = \frac12\Vert A^\frac12 u \Vert_X^2 + \frac12\Vert\dot{u}\Vert_X^2,
\end{displaymath}
it is natural to consider the first-order form of \eqref{eq:second_order},
\begin{equation}
    \frac{d}{dt} 
    \begin{pmatrix}
        u \\ \dot{u}
    \end{pmatrix}
    =
    \begin{pmatrix}
        0 & \mathrm{id}_X \\
        -A & -B
    \end{pmatrix}
        \begin{pmatrix}
        u \\ \dot{u}
    \end{pmatrix}
    +
    \begin{pmatrix}
        0 \\
        f(u,\dot{u})
    \end{pmatrix}
    ,
    \label{eq:first_order_1}
\end{equation}
on the space $Y =\dom(A^\frac12) \times X$ equipped with the inner product 
\begin{equation}
    \left\langle \begin{pmatrix}
        u_1 \\ u_2
    \end{pmatrix}, \begin{pmatrix}
        v_1 \\ v_2
    \end{pmatrix} \right\rangle_Y =
    \langle A^\frac12 u_1 , A^\frac12 v_1 \rangle_X + \langle u_2,v_2\rangle_X.
    \label{eq:norm_E}
\end{equation}
Denoting the linear and nonlinear parts of \eqref{eq:first_order_1} by $\mathcal{L}_B$ and $\mathcal{F}$, we have
\begin{equation}
    \frac{d}{dt} \xi = \mathcal{L}_B \xi + \mathcal{F}(\xi), \qquad \xi \in \dom(\mathcal{L}_B) \subset Y,
    \label{eq:first_order_2}
\end{equation}
where $\dom (\mathcal{L}_B) \supset \dom(A) \times \dom(B)$ (if $\alpha \in [\frac12,1]$ and $B= A^\alpha$, then equality holds).

Let us concentrate on the linear part of \eqref{eq:first_order_2} for the moment.
In the undamped case of $B =0$, $\mathcal{L}_0$ is skew-adjoint and hence generates a strongly continuous group of unitary operators on $Y$ \cite{chen1989proof}.
In view of the definition of energy through \eqref{eq:norm_E}, this corresponds to a conservation of energy statement.
Computed through the lens of \eqref{eq:first_order_1}, we have explicitly
\begin{equation}
    \frac{d}{dt} \left\Vert\begin{pmatrix}
        u(t) \\ \dot{u}(t)
    \end{pmatrix}\right\Vert^2_Y = 2\langle \dot{u}(t), \ddot{u}(t) + Au(t) \rangle_X = 0.  
    \label{eq:energy_cons}
\end{equation}
The addition of a positive self-adjoint dissipation $B$ modifies \eqref{eq:energy_cons} formally to
\begin{displaymath}
    \frac{d}{dt} \left\Vert\begin{pmatrix}
        u(t) \\ \dot{u}(t)
    \end{pmatrix}\right\Vert^2_Y  = 2 \langle \dot{u}(t), B \dot{u}(t) \rangle_X \leq 0.
\end{displaymath}
For $B$ strictly positive, the operator $\mathcal{L}_B$ is dissipative on $Y$ ($\langle \xi , \mathcal{L}_B \xi \rangle_Y \leq 0$ for all $\xi \in Y$) and densely defined.
Hence $\mathcal{L}_B$ is closable on $Y$ and, by the Lumer-Phillips theorem \cite{lumer1961dissipative}, its closure (denoted by $\mathcal{L}_B$ again) generates a strongly continuous semigroup of contractions on $Y$.

\subsection{Sectoriality, fractional powers}

For any operator $B$ satisfying \ref{A2}, the operator $\mathcal{L}_B$ enjoys some additional properties, as proven by \cite{chen1989proof}.
To be able to state their results, we recall the definition of a sectorial operator.

\begin{definition}
    Let $X$ be a Banach space and let $A$ denote a closed, densely defined linear operator on $X$ with spectrum $\sigma(A)$.
    $A$ is called sectorial if there exists $\omega \in (\pi/2,\pi)$ such that
\begin{displaymath}
    \Sigma_\omega = \{z \in \mathbb{C}\setminus\{0\} : |\arg z| < \omega\}
\end{displaymath}
is contained in the resolvent set $\rho(A) = \mathbb{C} \setminus \sigma(A)$ and
there exists $M>0$ such that
\begin{displaymath}
    \Vert(\lambda \, \mathrm{id}_X-A)^{-1}\Vert \leq \frac{M}{|\lambda|} \qquad \text{for all } \lambda \in \Sigma_\omega.
\end{displaymath}

\end{definition}

\begin{theorem}[Theorem~1.1, \cite{chen1989proof}]
    Assume \ref{A1}-\ref{A2}.
    Then $\mathcal{L}_B : \mathrm{dom}(\mathcal{L}_B) \to Y$ is sectorial and generates an analytic semigroup $t \mapsto e^{\mathcal{L}_B t}$.
    Moreover, there is a constant $ \sup \mathrm{Re} \, \sigma(\mathcal{L}_B) < \omega < 0$ such that
    \begin{displaymath}
        \Vert e^{\mathcal{L}_B t} \Vert_Y \leq e^{\omega t} , \qquad t \geq 0.
    \end{displaymath}
\end{theorem}

The property of sectoriality provides a smoothing effect for the semigroup, which we shall later utilize to permit the nonlinearity $\mathcal{F}$ to be unbounded on $Y$.
In particular, $\mathcal{F}$ will be permitted to contain derivative operators.
To be able to characterize the permitted domain of $\mathcal{F}$ more precisely, we recall the following notion of fractional power spaces.

\begin{definition}[Fractional powers of sectorial operators] 
\label{def:fractpowers}
    Suppose $L$ is a sectorial operator on a Banach space $Y$ and $\mathrm{Re} \, \sigma(L) < 0$.
    Let $\Gamma$ denote the Euler-$\Gamma$ function 
    \begin{displaymath}
        \Gamma(z) = \int_0^\infty t^{z-1} e^{-t} dt, \qquad \mathrm{Re} \, z > 0.
    \end{displaymath}
    Then for any $\beta > 0$,
    \begin{displaymath}
        (-L)^{-\beta} := \frac{1}{\Gamma(\beta)} \int_0^\infty t^{\beta-1} e^{Lt} dt
    \end{displaymath} 
    defines a bounded linear operator on $Y$ that is injective. 
For positive powers, set $\mathrm{dom} \left( (-L)^{\beta}\right) := \mathrm{im} \left( (-L)^{-\beta} \right)$ and
\begin{displaymath}
    (-L)^{\beta} : = \left((-L)^{-\beta} \right)^{-1}. 
\end{displaymath}
The operator $(-L)^{\beta} $ is closed, thus the space $\mathrm{dom} \left( (-L)^{\beta}\right)$ endowed with the graph norm defines a Banach space.
\end{definition}

We shall henceforth use the notation
\begin{displaymath}
    Y^\beta := \mathrm{dom} \left( (-\mathcal{L}_B)^{\beta}\right),
\end{displaymath}
equip $Y^\beta$ with the graph norm for $\beta \in (0,1)$ and set $Y^0 = Y$.
The following assumption characterizes the permitted class of nonlinearities in \eqref{eq:second_order}.

\begin{enumerate}[label =(A.\arabic*)] 
    \setcounter{enumi}{2}
    \item \label{A3x} \textbf{(Nonlinearity)} 
    Fix $\beta \in [0,1)$.
    The map $\mathcal{F} :  Y^\beta \to Y$ is of class $C^r$, $r \in \mathbb{N} \cup \{\infty\}$, and $\mathcal{F}(\Vert \xi \Vert_{Y^\beta}) = o(\Vert \xi \Vert_{Y^\beta})$ as $\Vert \xi \Vert_{Y^\beta} \to 0$.
\end{enumerate}

\begin{remark}
    The spaces $\{Y^\beta \}_{\beta \in (0,1)}$ are proper subsets of $Y$ and $Y^\beta \subset Y^{\beta'}$ is continuously embedded for $\beta' < \beta $; hence \ref{A3x} is more permissive for higher $\beta$. 
    In particular, for PDEs, derivatives are permitted in the nonlinearity so long as, qualitatively speaking, 
    the order of derivatives appearing in the nonlinearity is strictly less than that appearing in the linear part.
    Appendix~\ref{appendix:domLB} makes this statement more precise for the case $B = 2\mu A^\frac12$.
\end{remark}

\subsection{Existence of a semiflow}

We recall the existence, smoothness and uniqueness of solutions for \eqref{eq:first_order_2}.

\begin{theorem}[Corollary 3.4.6, \cite{henry1981geometric}] \label{thm:exist}
    Assume that \ref{A1}-\ref{A3x} hold.
    Then equation \eqref{eq:first_order_2} generates a semiflow $\varphi : \mathcal{D}^\varphi \to Y^\beta$ on the open domain $\mathcal{D}^\varphi \subset [0,\infty) \times Y^\beta$.
    The semiflow $\varphi$ is jointly $C^r$ on $\mathcal{D}^\varphi \cap (0,\infty) \times Y^\beta$.
    Moreover, $\varphi$ satisfies the integral equation
    \begin{equation}
        \varphi_t(\xi) = e^{\mathcal{L}_B t} \xi + \int_0^t e^{\mathcal{L}_B (t-s)} \mathcal{F} \circ \varphi_s (\xi) \, ds.
        \label{eq:int_eq}
    \end{equation}
\end{theorem}

\begin{proof}
    Local existence and uniqueness of solutions follows from \cite[Theorem~3.3.3]{henry1981geometric}, smoothness in turn follows from \cite[Corollary~3.4.6]{henry1981geometric}.
    The representation \eqref{eq:int_eq} is given by \cite[Lemma~3.3.2]{henry1981geometric}.
\end{proof}

\subsection{Statement of the main result}
\label{sect:statement}

Next, we formulate our main result.
For this, we first introduce the following assumption on the spectrum of $\mathcal{L}_B$.

\begin{enumerate}[label =(A.\arabic*)] 
    \setcounter{enumi}{3}
    \item \label{A4} \textbf{(Spectrum)} 
    Let $\Sigma \subset \sigma(\mathcal{L}_B)$ be a bounded, isolated, nonempty spectral subset of the form
    \begin{displaymath}
        \Sigma = \{ \lambda \in \sigma(\mathcal{L}_B) \; | \; \mathrm{Re} \, \lambda \geq \gamma \}
    \end{displaymath}
    for some $\gamma \in \mathbb{R}$, and set $\Sigma' := \sigma(\mathcal{L}_B) \setminus \Sigma$.
    Let $\gamma^+ :=\inf_{\lambda \in \Sigma} \mathrm{Re} \, \lambda,$ and $\gamma^- := \sup_{\lambda \in {\Sigma'}} \mathrm{Re} \, \lambda$.
\end{enumerate}

In most cases of interest, $\gamma$ in \ref{A4} can be chosen arbitrarily below $\sup \mathrm{Re} \, \sigma(\mathcal{L}_B)$. 
Indeed, if $\alpha \in [\frac12,1)$, then $\mathcal{L}_B$ has compact resolvent for $B$ satisfying \ref{A2}\footnote{
The fact that $\mathcal{L}_{\mu A^\alpha}$ has compact resolvent  for $\alpha \in [\frac12,1)$ is stated explicitly in \cite[Appendix~A]{chen1989proof}.
Otherwise, one may utilize the representations (1.12) and (1.14) therein for the resolvents of $\mathcal{L}_{\mu A^\alpha}$ and $\mathcal{L}_B$ together with formula (4.5) therein to reach the same conclusion for $\mathcal{L}_B$.
},
i.e., $\sigma(\mathcal{L}_B)$ consists solely of eigenvalues of finite algebraic multiplicity that form a discrete subset of $\mathbb{C}$.
Defining the associated projection map $P_\Sigma$ via Dunford's integral (see, e.g., \cite[Definition~5.24]{buhler2018functional}), we obtain the decomposition
\begin{displaymath}
    Y^\beta = Y^\beta_\Sigma \oplus Y^\beta_{\Sigma'}:= \mathrm{im}(P_\Sigma) \oplus \ker(P_\Sigma: Y^\beta \to Y^\beta).
\end{displaymath}
Note that as a set, $Y^\beta_\Sigma $ does not depend on $\beta$, since $\mathrm{im}(P_\Sigma) \subset \dom(\mathcal{L}_B)$.
Note also that in the case when $\mathcal{L}_B$ has compact resolvent,  $Y_\Sigma^\beta$ is finite dimensional for any choice of $\gamma < \sup \mathrm{Re} \, \sigma(\mathcal{L}_B) < 0$.

For a given set $U \subset Y^\beta$, let us denote by $\mathcal{U}(U)$ the set of pairs $(t,\xi)$ for which a trajectory initiated at $z$ remains in $U$ for all times up to $t$, i.e.,
\begin{displaymath}
    \mathcal{U}(U) = \left\{ (t,\xi) \in \mathcal{D}^\varphi  \; \big| \; \varphi ([0,t] \times \{\xi\}) \subset U  \right\},
\end{displaymath}
and let $\mathcal{U}_\xi(U) = \{ t \in [0,\infty) \; | \; (t,\xi) \in \mathcal{U}(U) \}$.
By Lyapunov stability of the origin, there exists an open neighborhood $O \subset U$ of $0$ such that $[0,\infty)_t \times O \subset \mathcal{U}(U)$.

\begin{theorem}[Existence of slow SSMs] \label{thm:main}
    Assume \ref{A1}-\ref{A4} with either $\beta = 0$ (in \ref{A3x}) or $Y^\beta_\Sigma$ finite dimensional.
    The following hold:
    \begin{enumerate}[label=\upshape{(\roman*)}]
        \item \label{thm1st1} \textup{\textbf{(Existence)}} There exists an open neighborhood $U \subset Y^\beta$ of $0$ and a $C^1$ submanifold $W^\Sigma \subset U$ tangent to $Y^\beta_\Sigma $ at $0$, which is locally invariant under $\varphi_t$: If $\xi \in W^\Sigma$, then $\varphi_t(\xi) \in W^\Sigma$ for all $(t,\xi) \in \mathcal{U}(U)$.
        $\{\varphi_t|_{W^\Sigma} \}_{t \geq 0}$ extends to a jointly $C^1$ flow on $\mathcal{U}(U \cap W^\Sigma)$.
        \item \label{thm1st2} \textup{\textbf{(Smoothness)}} If $\gamma^- < \ell \gamma^+ $ for some positive integer $\ell \leq r$, then $W^\Sigma$ is a $C^\ell$ submanifold of $U$. 
        In either case, the restricted semiflow $\varphi_t|_{W^\Sigma}$ extends to a jointly $C^\ell$ (resp.\ $C^k$) flow on $\mathcal{U}(U \cap W^\Sigma)$.
        \item \label{thm1st3} \textup{\textbf{(Attractivity)}} There exists a continuous map $\pi : U \to W^\Sigma$ such that 
        \begin{displaymath}
            \pi \circ \varphi_t(\xi) = \varphi_t\vert_{W^\Sigma} \circ \pi (\xi), \qquad  (t,\xi) \in \mathcal{U}(U).
        \end{displaymath}
        Moreover, for any $\xi \in U$, there exists $C > 0$ such that 
        \begin{displaymath}
            \Vert \varphi_t(\xi) - \varphi_t \circ \pi (\xi) \Vert_Y \leq C e^{ \gamma t}, \qquad  t \in \mathcal{U}_\xi(U).
        \end{displaymath}
        The preimages $\{\pi^{-1}(\zeta)\}_{\zeta \in W^\Sigma}$ are Lipschitz submanifolds of $U$, uniquely determined on $O$ by the above properties.
        \item \label{thm1st4} \textup{\textbf{(Pseudo-uniqueness)}} If $W^\Sigma$ and $\widetilde{W^\Sigma}$ are two invariant manifolds tangent to $Y^\beta_\Sigma $ at $0$, there exists a neighbourhood $V$ on which the reduced dynamics are topologically conjugate, i.e., there exists a homeomorphism $g : W^\Sigma \cap V \to \widetilde{W^\Sigma} \cap V$ such that
        \begin{displaymath}
            \varphi_t \circ g(\xi) = g\circ \varphi_t (\xi), \qquad \text{for all }  (t,\xi) \in \mathcal{U}(W^\Sigma \cap V ).
        \end{displaymath}
    \end{enumerate}
\end{theorem}

\begin{proof}
This is an application of \cite[Theorem~1.1]{chen1997invariant} to the integral equation \eqref{eq:int_eq} defining the semiflow $\varphi$. 
To localize the argument, one has to modify the nonlinearity in order to satisfy the Lipschitz constant requirement of \cite[Theorem~1.1]{chen1997invariant}.
We shall give more details of this part. 

In particular, for a cutoff function $\chi :Y^\beta \to \mathbb{R}$ having support in a neighborhood of $0$, we declare $ \varphi_t^\chi$ to be the semiflow determined by
\begin{equation}
    \varphi_t^\chi(\xi) = e^{\mathcal{L}_B t} \xi + \int_0^t e^{\mathcal{L}_B (t-s)} \chi(\varphi_s(\xi)) \mathcal{F} \circ \varphi_s^\chi (\xi) \, ds.
    \label{eq:semiflow_cutoff}
\end{equation}
Since $\varphi \equiv \varphi^\chi$ on a small enough neighborhood of the origin, the local results claimed in the statement can be concluded by considering \eqref{eq:semiflow_cutoff} in place of \eqref{eq:int_eq}.
Note that in general, there is no smooth choice for $\chi$. 
It can always be arranged to be Lipschitz continuous, hence \cite[Theorem~1.1]{chen1997invariant} yields a Lipschitz manifold $W^\Sigma$ and continuous foliation map $\pi$ as in the statement of the present theorem.
If $\beta = 0$, then the phase space $Y$ is a Hilbert space and $\chi$ can be chosen to be $C^\infty$ -- thus we obtain the usual smoothness results in \ref{thm1st1}-\ref{thm1st2} by, e.g., appealing to \cite{irwin1980new,de1995irwin}.
If $\beta \neq 0$, we may proceed with the construction customary in the delay equations literature \cite[Section~IX.4]{diekmann2012delay}, which utilizes the finite dimensionality of $Y^\beta_\Sigma$ (assumed here) to construct a cut-off function $\chi$ which retains its smoothness on a neighborhood containing the Lipschitz $W^\Sigma$ obtained via \cite[Theorem~1.1]{chen1997invariant} for $\varphi^\chi$.
Then, the smoothness results \ref{thm1st1}-\ref{thm1st2} can be recovered as before, noting that the proofs of \cite{irwin1980new,de1995irwin} only use smoothness of the nonlinearity about points on $W^\Sigma$.

The uniqueness statement for the foliation in \ref{thm1st3} follows from the Lyapunov stability of the fixed point and \cite[Lemma 4.6 and Remark 4.11]{buza2024spectral}.

The final statement, \ref{thm1st4}, is derived from the results of \cite{burchard1992smooth} for center manifolds; the procedure is detailed for the case of pseudo-unstable manifolds in \cite[Lemma~4.8]{buza2024spectral}. 
\end{proof}

\begin{remark}[The role of sectoriality and having compact resolvent]
    As previously mentioned, in the range $\alpha \in [\frac12,1)$ and under the assumptions \ref{A1}-\ref{A2}, the operator $\mathcal{L}_B$ is sectorial and has compact resolvent. 
    The role of these, informally, is as follows
    \begin{itemize}
        \item \textbf{(Sectoriality)} Permits the presence of spatial derivative terms in the nonlinearity $f$.
        \item \textbf{(Compact resolvent)} Permits the arbitrary choice of the spectral subset $\Sigma  = \{ \lambda \in \sigma(\mathcal{L}_B) \; | \; \mathrm{Re} \, \lambda \geq \gamma \}$ in \ref{A4}.
    \end{itemize}
    Neither of these two are, however, strictly necessary to apply the results of \cite{chen1997invariant}. 
    For instance, said results are applicable to the example of \cite{kogelbauer2018rigorous} as well, which was carefully constructed  \textit{not} to be sectorial and be reversible instead.
    In this sense, the example therein can be considered a special case, where in fact the stronger results of \cite{Cabre2003a} also hold, and provide uniqueness of $W^\Sigma$ in a sufficiently smooth subclass.
\end{remark}

\begin{remark}[SSM associated to an arbitrary pair of complex conjugate eigenvalues]
    An SSM associated to an arbitrary pair of complex conjugate eigenvalues $\{\lambda^+,\lambda^- \}$ can be constructed by intersecting the slow manifold $W^\Sigma$ from Theorem~\ref{thm:main} associated to the spectral subset $\Sigma = \{ \lambda \in \sigma(\mathcal{L}_B) \; | \; \mathrm{Re} \, \lambda \geq \gamma \}$ for $\gamma = \mathrm{Re} \, \lambda^+$ with the ($C^r$) strong stable manifold $W^{ss}$ associated to the spectral subset $\Sigma^{ss} = \{ \lambda \in \sigma(\mathcal{L}_B) \; | \; \mathrm{Re} \, \lambda \leq \gamma \}$.
    This construction only yields the desired manifold whenever $\sigma(\mathcal{L}_B) \cap (\{\gamma\} \times i \mathbb{R}) = \{\lambda^+,\lambda^- \}$, which can be assured generically in most scenarios of interest.
    For instance, when $B = A^\frac12$, this property is guaranteed so long as the eigenvalues of $A$ are simple (c.f.\ Figure~\ref{fig:1}).
    Construction of even more general manifolds is also permitted by replacing $W^{ss}$ above with the manifolds from \cite{de1997invariant} or \cite{buza2025smooth}.
    This has been performed in \cite[Theorem~15]{buza2025existence}; we omit the details here.
\end{remark}

\subsection{The periodically-forced case}

We also consider periodically forced perturbations of \eqref{eq:first_order_2}.
We make the following assumption on the forcing
\begin{enumerate}[label =(A.\arabic*)] 
    \setcounter{enumi}{4}
    \item \label{A5} \textbf{(Forcing)} 
    We suppose $\mathcal{H}: \mathbb{R} \to Y$ is a locally Hölder continuous map satisfying $\mathcal{H}(t+p)=\mathcal{H}(t)$ for some $p>0$ and all $t \in \mathbb{R}$
\end{enumerate}

The forced system we consider reads
\begin{equation}
    \frac{d}{dt} \xi = \mathcal{L}_B \xi + \mathcal{F}(\xi) + \varepsilon\mathcal{H}(t), \qquad \xi \in \dom(\mathcal{L}_B) \subset Y,
    \label{eq:first_order_3}
\end{equation}
for some $\varepsilon>0$.
By assumption \ref{A5},  the existence of a non-autonomous semiflow $\{\varphi_{t_0}^t\}_{t \geq t_0}$ with domain $\mathcal{D}^\varphi$ associated to \eqref{eq:first_order_3} ensues by the same results of \cite{henry1981geometric} quoted in Theorem~\ref{thm:exist}.

We suppose moreover the following.
\begin{enumerate}[label =(A.\arabic*)] 
    \setcounter{enumi}{5}
    \item \label{A6} \textbf{(Reference solution)} 
    Equation \eqref{eq:first_order_3} admits a time-periodic solution $\xi_*$; $\xi_*(t+p) = \xi_*(t)$ for all $t \in \mathbb{R}$.
\end{enumerate}

Assumption \ref{A6} can be guaranteed to hold under \ref{A1}-\ref{A3x} and \ref{A5} for $\varepsilon>0$ sufficiently small by the implicit function theorem (see \cite[Theorem~8.3.1]{henry1981geometric} for the details of this argument).

We state the time-periodic equivalent to the spectral assumption \ref{A4} in the language of Floquet theory. 
In terms of $\zeta = \xi - \xi_*(t)$, \eqref{eq:first_order_3} reads
\begin{displaymath}
    \frac{d}{dt} \zeta = \mathcal{L}_B \zeta + \mathcal{F}(\zeta) - \mathcal{F}(\xi_*(t)).
\end{displaymath}
Its linearization, the linear time-periodic equation\footnote{Here, the Fréchet derivative is taken in the sense $Y^\beta \to Y$, in accordance with assumption \ref{A3x}.}
\begin{displaymath}
    \frac{d}{dt}\zeta = \mathcal{L}_B \zeta +  D \mathcal{F}(\xi_*(t)) \zeta, \qquad \zeta \in Y^\beta
\end{displaymath}
generates a strongly continuous evolution family $\{T_{t_0}^t\}_{t \geq t_0}$ \cite[Theorem~7.1.3]{henry1981geometric} (this corresponds to the linearization of the semiflow $\{\varphi_{t_0}^t\}_{t \geq t_0}$ about $\xi_*$).
Denote by $\Phi(t) := T^{t+p}_{t}$ the period-$p$ map of the linear system.
Floquet theory implies that $\sigma(\Phi):= \sigma(\Phi(t))$ is independent of $t \in \mathbb{R}$ (see \cite[Lemma~7.2.2]{henry1981geometric}).
We make to following assumption on the spectrum:
\begin{enumerate}[label =(A.\arabic*)] 
    \setcounter{enumi}{6}
    \item \label{A7} \textbf{(Spectrum)} 
    Let $\Sigma \subset \sigma(\Phi)$ be a bounded, isolated, nonempty spectral subset of the form
    \begin{displaymath}
        \Sigma = \{ \lambda \in \sigma(\Phi) \; | \; |\lambda | \geq \gamma \}
    \end{displaymath}
    for some $\gamma >0$, and set $\Sigma' := \sigma(\Phi) \setminus \Sigma$.
    Let $\gamma^+ :=\inf_{\lambda \in \Sigma} |\lambda|$ and $\gamma^- := \sup_{\lambda \in {\Sigma'}} |\lambda|$.
\end{enumerate}

For $\varepsilon >0 $ small enough, it suffices to verify this condition on the autonomous generator $\mathcal{L}_B$, in the sense of \ref{A4}, for then perturbation theory \cite{kato2013perturbation} furnishes \ref{A7}.

From here on, we work with the coordinate $\zeta$ exclusively, meaning that the $p$-periodic solution $\xi_*$ is now attained at $0$.
Assumption \ref{A7} implies, by \cite[Theorem~7.2.3]{henry1981geometric}, the existence of a time-dependent, $p$-periodic invariant decomposition
\begin{displaymath}
    Y^\beta = Y^\beta_\Sigma(t) \oplus Y^\beta_{\Sigma'}(t).
\end{displaymath}

We are now ready to state the main theorem in the time-periodic case.
To keep the statement concise, we set
\begin{displaymath}
    \mathcal{U}(U) = \big\{ (t,s,u) \in \mathcal{D}^\varphi \; \big| \; \varphi_s^{[s,t]}(u) \subset U \big\},
\end{displaymath}
and $\mathcal{U}_{s,u}(U) = \{t \in [s,\infty) \; | \; (t,s,u) \in \mathcal{U}(U)\}$.

\begin{theorem}[Existence slow SSMs in the periodically forced case] \label{thm:periodic}
Assume \ref{A1}-\ref{A3x}, \ref{A5}-\ref{A7}.
The following hold.
\begin{enumerate}[label=\upshape{(\roman*)}]
\item \label{st1} \textbf{\textup{(Existence)}} There exists an open neighbourhood $U \subset Y^\beta$ of the origin and a continuous family of $C^1$ submanifolds $W^\Sigma_t \subset U$ tangent to $Y^\beta_\Sigma(t)$ at $0$, which is locally invariant under $\varphi$: If $u \in W^\Sigma_s$, then $\varphi_s^t(u)\in W^\Sigma_t$ for all $(t,s,u) \in \mathcal{U}(U)$.
\item \label{st2} \textbf{\textup{(Smoothness)}}  If $\gamma^- < ( \gamma^+)^\ell $ for some positive integer $\ell \leq r$, then  $W^\Sigma_t$ is $C^\ell$ smooth for all $t \in \mathbb{R}$.
\item \label{st3} \textbf{\textup{(Attractivity)}} There exist a continuous maps $\pi_t : U \to W^\Sigma_t$, $t \in \mathbb{R}$, such that 
\begin{displaymath}
    \pi_t \circ \varphi_s^t(u) = \varphi_s^t( \pi_s (u)), \qquad  (t,s,u) \in \mathcal{U}(U).
\end{displaymath}
Moreover, for any $u \in U$ and $s \in \mathbb{R}$, there exists $C > 0$ such that 
\begin{displaymath}
    | \varphi_s^t(u) -  \varphi_s^t( \pi_s (u))| < C e^{  (t-s) \ln(\gamma)/p}, \qquad \text{for all } t \in \mathcal{U}_{s,u}(U).
\end{displaymath}
The preimages $\{\pi_t^{-1}(w)\}_{w \in W^\Sigma_t}$ are Lipschitz submanifolds foliating $U$ for each $t \in \mathbb{R}$.
\end{enumerate}
\end{theorem}

\begin{proof}
    The steps required to extend a result such as Theorem~\ref{thm:main} to the time-periodic case are well understood and fairly standard, we merely outline them here.
    One considers the Poincaré map $\varphi_t^{t+p}$ on a neighborhood of the origin (see \cite[Section~8.4]{henry1981geometric}).
    The assumptions permit the application of \cite[Theorem~3.1]{chen1997invariant} to  $\varphi_t^{t+p}$, which yields a $C^\ell$ manifold $W^\Sigma_t$ invariant under $\varphi_t^{t+p}$. 
    The obtained manifolds can be glued together for all $t \in \mathbb{R}$ to form a continuous (in time) family of $C^\ell$ manifolds $t \mapsto W_t^\Sigma$ that satisfy $\varphi_{s}^t(W_{s}^\Sigma) \subset W_t^\Sigma$, just as in the proof of \cite[Theorem~2]{abbasciano2026data2}.
\end{proof}

\section{Computation of SSMs for PDEs}
\label{sect:comp}

In this section, we describe a general procedure to obtain approximations, through analytical means, to  the manifolds $W^\Sigma$ obtained in Theorem~\ref{thm:main}.
As a preliminary step, we require detailed information about spectral properties of the linear part.
We restrict ourselves to the setting when $B = 2\mu A^\alpha$, $\alpha \in [0,1]$, by reasons motivated in the following sections, specifically Section~\ref{sect:beam_damping}.
For ease of notation, and since this is the most important case, we shall also only detail the case when the slow spectral subset $\Sigma$ consists of a single pair of complex conjugate eigenvalues, but the methodology could be extended to cover larger spectral subsets in a straightforward manner.

\subsection{Spectral properties of \texorpdfstring{$\mathcal{L}_{2\mu A^\alpha}$}{}}
\label{sect:spectrum}

We recollect some results from \cite[Appendix A, Lemmas A.1 and A.2]{chen1989proof} that are pertinent to our calculations.

We shall assume throughout this section $\alpha \in [0,1]$, $ \mu \in (0,1)$, with the most relevant case for our subsequent applications herein being $\alpha = \frac12$.
Let $\{\sigma_n\}_{n \in \mathbb{N}} \subset \mathbb{R}$ denote the eigenvalues of $A$ ordered in ascending order (all assumed to be simple\footnote{The assumption of simplicity of eigenvalues is not a strict requirement here; it is assumed merely for convenience and ease of exposition.}), and let $\{\phi_n\}_{n \in \mathbb{N}}$ denote the corresponding eigenvectors, normalized according to \eqref{eq:phi_normal}.

For $\alpha < 1$, the operator $\mathcal{L}_{2\mu A^\alpha}$ has compact resolvent.
The eigenvalues $\{ \lambda^\pm_n\}_{n \in \mathbb{N}}$ of $\mathcal{L}_{2\mu A^\alpha}$ are the solutions of
\begin{equation}
    \lambda^2 + 2\mu \sigma_n^\alpha \lambda + \sigma_n = 0, \qquad \text{that is,} \qquad \lambda^\pm_n = \left(-\mu \pm \sqrt{\mu^2-\sigma_n^{1-2\alpha}} \right) \sigma_n^\alpha.
    \label{eq:lambdas}
\end{equation}
The corresponding normalized eigenvectors $\{ \psi_n^\pm\}_{n \in \mathbb{N}}$ of $\mathcal{L}_{2\mu A^\alpha}$  on $Y$ are 
\begin{displaymath}
    \psi_n^+ = \begin{pmatrix}
        \phi_n \\ \lambda_n^+ \phi_n
    \end{pmatrix}, \qquad
    \psi_n^- = k_n\begin{pmatrix}
        \phi_n \\ \lambda_n^- \phi_n
    \end{pmatrix},
\end{displaymath}
where $k_n^2 = \frac{\sigma_n+|\lambda^+_n|^2}{\sigma_n+|\lambda^-_n|^2}$ (for $\alpha = \frac12$ we have $k_n =1$ for all $n$), and the $\phi_n$ are normalized according to
\begin{equation}
    (\sigma_n+|\lambda^+_n|^2) \Vert \phi_n\Vert_X^2 = 1, 
    \label{eq:phi_normal}
\end{equation}
so that $\Vert \psi_n^\pm \Vert_Y = 1$.

Suppose now that $\mu^2 \neq \sigma_n^{1-2\alpha}$ for all $n \in \mathbb{N}$ (for $\alpha = \frac12$ this is given).
The eigenvectors of the adjoint $\mathcal{L}_{2 \mu A^\alpha}^*$ are
\begin{displaymath}
\psi^{*-}_m = \frac{1}{v_m^-} \begin{pmatrix}
        \phi_m \\ - \overline{\lambda_m^+} \phi_m
    \end{pmatrix}, \qquad
    \psi^{*+}_m = \frac{1}{v_m^+} \begin{pmatrix}
        \phi_m \\ - \overline{\lambda_m^-} \phi_m
    \end{pmatrix},
\end{displaymath}
corresponding to the eigenvalues $\overline{\lambda_m^+}$ and $\overline{\lambda_m^-}$, where
\begin{displaymath}
    v_m^+ = \frac{\sigma_m - (\overline{\lambda_m^-})^2}{\sigma_m + |\lambda_m^-|^2}, \qquad v_m^- = \frac{\sigma_m - (\overline{\lambda_m^+})^2}{\sigma_m + |\lambda_m^+|^2}.
\end{displaymath}
The $\{\psi^{*\pm}_m\}_{m \in \mathbb{N}}$ form a bi-orthogonal system with respect to the eigenvectors $\{\psi^\pm_m \}_{m \in \mathbb{N}}$, i.e.,\footnote{When complexifying $Y$, we take the second entry of $\langle \cdot,\cdot\rangle_Y$ to be anti-linear.}
\begin{align*}
    \langle \psi_m^{*-},\psi_n^{+} \rangle_Y &= \langle \psi_m^{*+},\psi_n^{-} \rangle_Y = \delta_{nm}, \\
    \langle \psi_m^{*+},\psi_n^{+} \rangle_Y &= \langle \psi_m^{*-},\psi_n^{-} \rangle_Y = 0, \qquad  n,m \in \mathbb{N}.
\end{align*}
Any $\xi \in Y$ can be written as
\begin{equation}
    \xi = \sum_{n=1}^\infty \langle \xi , \psi^{*-}_n \rangle_Y \psi_n^+ + \sum_{n=1}^\infty \langle \xi , \psi^{*+}_n \rangle_Y \psi_n^-,
    \label{eq:y_expansion}
\end{equation}
with sums converging in the norm $\Vert \cdot \Vert_Y$.

The sums in \eqref{eq:y_expansion} converge also in the $\Vert \cdot \Vert_{\mathcal{L}_{2 \mu A^\alpha}}$ norm for $\xi \in \dom(\mathcal{L}_{2 \mu A^\alpha})$, as seen from considering the expansion for $\mathcal{L}_{2 \mu A^\alpha} \xi \in Y$ according to \eqref{eq:y_expansion}
and noting that $\Vert \cdot \Vert_{\mathcal{L}_{2 \mu A^\alpha}}$ is equivalent to $\Vert \mathcal{L}_{2 \mu A^\alpha} \cdot \Vert_Y$.

\subsection{SSM parameterization}
\label{sect:param}

In solving the invariance equation defining the manifold, we follow closely the exposition of \cite{kogelbauer2018rigorous}, which in turn follows the guidelines placed by the parameterization method \cite{Cabre2003a}.
In particular, we shall solve the invariance equation,
\begin{equation}
    (\mathcal{L}_{2 \mu A^\alpha} + \mathcal{F}) \circ K = DK[R],
    \label{eq:invariance_param}
\end{equation}
up to a finite polynomial order,
for the unknown embedding $K: \Theta \to Y^\beta$ of the reduced space $\Theta$ (a finite-dimensional manifold) into the phase space $Y^\beta$; and for $R$, a vector field on $\Theta$ defining the reduced dynamics.
In our setting, we shall take $\Theta = Y^\beta_\Sigma$ and $K : Y^\beta_\Sigma \to Y^\beta$, where a manifold $W^\Sigma$ obtained from Theorem~\ref{thm:main} is envisioned as an image of a subset about the origin under $K$.

Equation \eqref{eq:invariance_param} carries with it a degree of freedom in that the reduced space is only defined up to a diffeomorphism.
For the details, we refer to \cite{haro2016parameterization}.
We will exploit this degree of freedom to insist that the reduced dynamics are modeled on $\Theta = \mathbb{C}$ instead of $Y^\beta_\Sigma$ 
and to
impose the form of the reduced dynamics as \eqref{eq:R_form} later on.

While the proof of existence of invariant manifolds with parameterization method hinges on the invertibility of $\mathcal{L}_{2 \mu A^\alpha}$, equation \eqref{eq:invariance_param} itself is solvable to finite polynomial order in the present setting, considering that each coefficient (of polynomial order $k$) will satisfy $\mathrm{im}(K^k) \subset \dom(\mathcal{L}_{2 \mu A^\alpha})$ recursively.
Moreover, the manifolds $W^\Sigma$ obtained via Theorem~\ref{thm:main} also satisfy $W^\Sigma \subset \dom(\mathcal{L}_{2 \mu A^\alpha})$.
In fact, it can be shown,
in a fashion analogously to \cite[Section~5]{buza2025existence}, that if Theorem~\ref{thm:main}\ref{thm1st2} holds with $\ell$, then the coefficients $K^k$\footnote{In comparison to \eqref{eq:K_expansion} below, $K^k(z,\overline{z}) = \sum_{|n| = k} K_n(z,\overline{z})^n$.} obtained from \eqref{eq:invariance_param} coincide with the Taylor coefficients of $W^\Sigma$ up to order $\ell$, expanded about $0$ as a graph over $Y^\beta_\Sigma$ (up to diffeomorphism invariance).

We begin by declaring $\Sigma := \{ \lambda_1^+,\lambda_1^- \}$ to be the spectral subset composed of the dominant modes, with $\lambda_n^\pm$ as in \eqref{eq:lambdas}.
The (complexified) spectral subspace associated to $\Sigma$ is 
\begin{displaymath}
    (Y^\beta_\Sigma)_\mathbb{C} = \mathrm{span}_\mathbb{C} \left\{ \psi_1^+ ,\psi_1^-\right\}.
\end{displaymath}
Given that $\psi_1^+$ and $\psi_1^-$ are complex conjugates to one another, if one defines the isomorphism $\varpi:\mathbb{C}^2 \to Y^\beta_\Sigma$ by $(z_1,z_2) \mapsto z_1\psi_1^+  + z_2 \psi_1^-$, then the real spectral subspace $Y^\beta_\Sigma$ is given by $\varpi(\Delta_c)$, where $\Delta_c := \{ (z,\overline{z}) \in \mathbb{C}^2 \; | \; z \in \mathbb{C} \}$.

Under the assumptions of Theorem~\ref{thm:main}, \eqref{eq:invariance_param} admits a $C^\ell$ solution pair $[\widetilde{K},\widetilde{R}]$ for $\Theta = Y^\beta_\Sigma$. 
For this expository section, we simply assume a spectral ratio of $\ell \geq 3$ in part \ref{thm1st2} of Theorem~\ref{thm:main}.
Rather than seeking the real embedding $\widetilde{K}$, we instead seek
\begin{displaymath}
    K: \Delta_c \xrightarrow{\varpi|_{\Delta_c}} Y^\beta_\Sigma \xrightarrow{\widetilde{K}} Y^\beta
\end{displaymath}
directly.
We approximate $K$ by Taylor expanding $\widetilde{K}$ up to order $\ell$, which yields
\begin{equation}
    K(z,\overline{z}) = \sum_{|n| =1}^\ell K_n (z,\overline{z})^n + o(|z|^\ell)
    \label{eq:K_expansion}
\end{equation}
as $|z| \to 0$,
for $n=(n_1,n_2) \in \mathbb{N}^2$ and coefficients $K_n \in (Y^\beta)_\mathbb{C}$.
These coefficients must satisfy $K_{(n_1,n_2)} = \overline{K}_{(n_2,n_1)}$.
For our calculations, it will also be convenient to extend $K^{\leq \ell}$ to $\hat{K}^{\leq \ell}:\mathbb{C}^2 \to (Y^\beta)_\mathbb{C}$;\footnote{The notation $(\cdot)^{\leq \ell}$ refers to the $\ell$-Taylor jet.} equation \eqref{eq:invariance_param} preserves the real structure so a subsequent restriction $K^{\leq \ell}=\hat{K}^{\leq \ell}|_{\Delta_c}$ will be real -- in particular, the coefficients $K_n$ must be the same for both maps.

We impose, once more using diffeomorphism invariance, the Hopf normal form for the reduced dynamics, so as to avoid working with large coefficients $K_{(1,2)}$, $K_{(2,1)}$ due to the near resonances $2\lambda_1^++\lambda_1^- \approx \lambda_1^+$ and $2\lambda_1^++\lambda_1^- \approx \lambda_1^+$, just as in \cite[page 1122]{kogelbauer2018rigorous}.\footnote{These near resonances occur when the damping is weak ($\mu \ll 1$), which is most often the case in real-world mechanical systems.}
In particular, we take for the reduced dynamics $\hat{R}$, a vector field on $\mathbb{C}^2$,
\begin{equation}
     \hat{R}(z_1,z_2) = \begin{pmatrix}
        \lambda_1^+ z_1 +R_0 z_1^2z_2 \\
        \lambda_1^- z_2 +\overline{R}_0 z_1z_2^2
    \end{pmatrix},
    \label{eq:R_form}
\end{equation}
for some $R_0 \in \mathbb{C}$ to be determined.

Inserting (the extension of) \eqref{eq:K_expansion} and \eqref{eq:R_form} into \eqref{eq:invariance_param} and truncating at order $\ell$, we obtain
\begin{multline}
    \sum_{|n| =1}^\ell \mathcal{L}_{2 \mu A^\alpha} K_n (z_1,z_2)^n - \sum_{|n| =1}^\ell (\lambda_1^+n_1+ \lambda_1^- n_2) K_n(z_1,z_2)^n  \\
    =- (\mathcal{F} \circ K)^{\leq \ell}(z_1,z_2)   + \sum_{|n| = 3}^\ell \big( R_0 (n_1-1) + \overline{R}_0 (n_2-1)\big) K_{n-(1,1)}(z_1,z_2)^n.
    \label{eq:invariance_ordered}
\end{multline}
(In the last term of \eqref{eq:invariance_ordered}, the coefficients $K_n$ with $n_1 <0$ or $n_2<0$ are taken to be $0$.)

Equation \eqref{eq:invariance_ordered} can now be solved recursively with respect to increasing polynomial order.
At order one, \eqref{eq:invariance_ordered} reduces to the eigenvalue problems
\begin{subequations} \label{eq:1storder}
\begin{align}
    \mathcal{L}_{2 \mu A^\alpha} K_{(1,0)} &= \lambda_1^+ K_{(1,0)}, \\
    \mathcal{L}_{2 \mu A^\alpha} K_{(0,1)} &= \lambda_1^- K_{(0,1)}.
\end{align}
\end{subequations}
Although technically \eqref{eq:1storder} admits a degree of freedom in rescaling the eigenvectors, we shall always take as solutions $K_{(1,0)} = \psi_1^+$, $K_{(0,1)} = \psi_1^-$ the unit norm eigenvectors.

At subsequent orders, the analysis depends on the specific problem at hand, due to the presence of the term $(\mathcal{F} \circ K)^{\leq \ell}(z_1,z_2)$ in \eqref{eq:invariance_ordered}, and hence is best described directly through the course of examples (see Sections \ref{sect:beam_mfd} and \ref{sect:plate_computation}).
Here, we merely outline the general procedure.
At order two, we have
\begin{subequations} \label{eq:2ndorder}
\begin{align}
    &\big[\mathcal{L}_{2 \mu A^\alpha} -  2\lambda_1^+ \big] K_{(2,0)} = - \frac12 D^2 \mathcal{F}(0)[\psi_1^+,\psi_1^+], \label{eq:2ndorder_first} \\
    &\big[\mathcal{L}_{2 \mu A^\alpha}- (\lambda_1^- + \lambda_1^+) \big]K_{(1,1)} = -  D^2 \mathcal{F}(0)[\psi_1^+,\psi_1^-], \\
    &\big[\mathcal{L}_{2 \mu A^\alpha} - 2\lambda_1^- \big] K_{(0,2)} = - \frac12 D^2 \mathcal{F}(0)[\psi_1^-,\psi_1^-].
\end{align}
\end{subequations}
Due to our standing assumption on $\ell$, the operators on the left hand sides are all invertible as linear maps $\dom(\mathcal{L}_{2 \mu A^\alpha}) \to Y$.
We project \eqref{eq:2ndorder_first} to each direction in the bi-orthogonal system recalled in Section~\ref{sect:spectrum}, to obtain, for $m \in \mathbb{N}$,
\begin{align*}
    \left\langle \big[\mathcal{L}_{2 \mu A^\alpha} -  2\lambda_1^+ \big] K_{(2,0)}  ,\psi_m^{*-} \right\rangle_Y &= \left\langle  K_{(2,0)},\mathcal{L}_{2 \mu A^\alpha}^* \psi_m^{*-} \right\rangle_Y - 2 \lambda_1^+\left\langle  K_{(2,0)},\psi_m^{*-} \right\rangle_Y \\
    &= \left\langle  K_{(2,0)},\overline{\lambda_m^+} \psi_m^{*-} \right\rangle_Y - 2 \lambda_1^+\left\langle  K_{(2,0)},\psi_m^{*-} \right\rangle_Y \\
    &= (\lambda_m^+ - 2 \lambda_1^+)\left\langle  K_{(2,0)},\psi_m^{*-} \right\rangle_Y \\
    &= -\left\langle  \frac12 D^2 \mathcal{F}(0)[\psi_1^+,\psi_1^+],\psi_m^{*-} \right\rangle_Y,
\end{align*}
which, combined with an analogous computation after applying $\langle \cdot, \psi_m^{*+} \rangle_Y$ to \eqref{eq:2ndorder_first}, yields
\begin{align*}
    &\left\langle  K_{(2,0)},\psi_m^{*-} \right\rangle_Y = \frac{1}{2 \lambda_1^+ - \lambda_m^+}\left\langle  \frac12 D^2 \mathcal{F}(0)[\psi_1^+,\psi_1^+],\psi_m^{*-} \right\rangle_Y, \\
    &\left\langle  K_{(2,0)},\psi_m^{*+} \right\rangle_Y = \frac{1}{2 \lambda_1^+ - \lambda_m^-}\left\langle  \frac12 D^2 \mathcal{F}(0)[\psi_1^+,\psi_1^+],\psi_m^{*+} \right\rangle_Y.
\end{align*}
The coefficient $K_{(2,0)}$ can then be reconstructed via the expansion \eqref{eq:y_expansion}, 
\begin{displaymath}
    K_{(2,0)} = \sum_{n=1}^\infty \langle K_{(2,0)} , \psi^{*-}_n \rangle_Y \psi_n^+ + \sum_{n=1}^\infty \langle K_{(2,0)}, \psi^{*+}_n \rangle_Y \psi_n^-,
\end{displaymath}
an expression that converges in $\dom(\mathcal{L}_{2 \mu A^\alpha})$.

At order three, completely analogous computations ensue, now with the additional term  depending on $R_0$ (the final one in \eqref{eq:invariance_ordered}) added to the right hand side of the equations for $K_{(2,1)}$ and $K_{(1,2)}$.

\subsection{Backbone curves}
\label{sect:backbone}

In structural vibrations, the damped natural frequencies, $\mathrm{Im}(\lambda_n^+)$, are the primary objects of study, due to their role in identifying the forcing regimes under which the structure, as modeled by a linear system, is most prone to large displacements.
For nonlinear systems, identifying these regimes is far more challenging of a task, since 
the critical forcing frequencies (i.e., for which a local maximum of response amplitude is attained) now depend on the forcing amplitude.
The curves describing this dependence in the frequency -- response amplitude plane are usually termed \textit{backbone curves}.
They each correspond to a natural frequency of the linear system, in that the curves emanate from the points $(\mathrm{Im}(\lambda_n^+),0)$ attained in the zero-amplitude limit, $n \in \mathbb{N}$.
Of these, the most damaging is the least damped one, with $n =1$.

SSMs provide a convenient way to interpret backbone curves \cite{breunung2018explicit}.
To describe this in the present setting at the level of autonomous dynamics, let us transform the reduced dynamics \eqref{eq:R_form} to polar coordinates, by letting $z_1 = re^{i \theta}$, which gives
\begin{align*}
    \dot{r} &= \mathrm{Re} (\lambda_1^+)r + \mathrm{Re} (R_0)r^3,\\
    \dot{\theta} &=\mathrm{Im} (\lambda_1^+) + \mathrm{Im}(R_0)r^2=:\vartheta(r).
\end{align*}
The backbone curve for the dominant frequency $\mathrm{Im}(\lambda_1^+)$ is defined as the
image of the map $\widetilde{\mathcal{B}}: [0,\infty)_r \to \mathbb{R} \times [0,\infty)$ given by
\begin{equation}
    \widetilde{\mathcal{B}}(r) = \big(\vartheta(r),r\big).
    \label{eq:backbone0}
\end{equation}
Defined this way, it is shown in \cite[Section~4]{breunung2018explicit} that this curve  in fact corresponds to the frequency-peak amplitude curve under the addition of forcing.
The particular forcing, however, generally necessitates $r$ (and the SSM) to be re-defined.
Nonetheless, this provides an interpretation of backbone curves  purely based on the autonomous dynamics, i.e., intrinsic to the structure at hand.

For practical purposes, it is more convenient to replace \eqref{eq:backbone0} with
\begin{equation}
    {\mathcal{B}}(r) = \big(\vartheta(r),\mathrm{Amp}(r)\big),
    \label{eq:backbone}
\end{equation}
for some physically relevant/observable quantity $\mathrm{Amp}$, here taken to be (from \cite{kogelbauer2018rigorous})
\begin{equation}
    \mathrm{Amp}(r) = \sqrt{\frac{1}{2\pi} \int_0^{2\pi}\left\Vert K(re^{i\theta},re^{-i\theta}) \right\Vert_Y^2 d\theta}.
    \label{eq:Amp}
\end{equation}
The addition of forcing will perturb $K$ and hence $\mathrm{Amp}$ slightly (see Section~\ref{sect:beam_forcing}).

\section{Example 1: SSM reduction of an Euler-Bernoulli beam}
\label{sect:beam}

Consider a finite length elastic beam modeled on the interval $[0,l]$.
Let $\rho>0$ denote its mass density per unit length, and let $EI>0$ denote the second moment of the modulus of elasticity.
We assume both of these to be constant for the sake of convenience, as it makes the exact identification of $\dom(A^\frac12)$ below easier.
Let $X := L^2([0,l])$ with $\langle w,v\rangle_X := \langle w, \rho v \rangle_{L^2([0,l])} $.

 In what follows, we shall survey several damping types (Section~\ref{sect:beam_damping}) and nonlinearities (Section~\ref{sect:beam_nonlin}) to gauge the range of models that fit our setup described in Section~\ref{sect:setup}.
In the end, we only perform explicit calculations in Section~\ref{sect:beam_mfd} for a single nonlinear model, which is described in Section~\ref{sect:beam_final_model} (the reader interested only in the physical application may immediately skip to these sections).

\subsection{Elastic operator, boundary conditions}
\label{sect:beam_A}

Let $A : \dom(A) \to X$ denote the elastic operator resulting from Euler-Bernoulli beam theory:
\begin{equation}
    Aw = \frac{1}{\rho} \frac{\partial^2}{\partial x^2} \left( EI \frac{\partial^2 w}{\partial x^2} \right),
    \label{eq:elastic_A}
\end{equation}
where 
\begin{equation}
    \dom(A) = \left\{ w \in H^4([0,l]) \; \Big| \; \eta(w) = 0, \; \forall \eta \in  F \subset H^{-4}([0,l])  \right\},
    \label{eq:dom(A)_beam}
\end{equation}
with 
$F \subset H^{-4}([0,l])$ denoting the set of boundary conditions, i.e., a four-element subset of bounded linear functionals on $H^4([0,l])$ 
for which $A$ is positive self-adjoint:
\begin{equation}
    \langle v,Aw \rangle_X - \langle Av,w \rangle_X = 0, \qquad  \qquad \langle w ,A w \rangle_X > 0, \qquad  v,w \in \dom(A).
    \label{eq:SAcond}
\end{equation}
Written out, these translate to
\begin{subequations} \label{eq:conds_on_BCs}
\begin{align}
    \left[ v \frac{\partial}{\partial x} \left( EI \frac{\partial^2w}{\partial x^2} \right) -  \frac{\partial v }{\partial x} EI \frac{\partial^2w}{\partial x^2}  \right]^l_0 - \left[ w\frac{\partial}{\partial x} \left( EI \frac{\partial^2v}{\partial x^2} \right) -  \frac{\partial w }{\partial x} EI \frac{\partial^2v}{\partial x^2}  \right]^l_0 &= 0, \\
    \left[ w \frac{\partial}{\partial x} \left( EI \frac{\partial^2w}{\partial x^2} \right) -  \frac{\partial w }{\partial x} EI \frac{\partial^2w}{\partial x^2}  \right]^l_0 + \int_0^l EI \left| \frac{\partial^2w}{\partial x^2}\right|^2 dx &> 0 .\label{eq:SAcond_2}
\end{align}
\end{subequations}
For instance, as a sample set of boundary conditions, we may impose
\begin{subequations} \label{eq:BCs_beam}
    \begin{align}
        &w(t,0) = 0, & & \frac{\partial^2 w}{\partial x^2}(t,0) = 0, \\
        & w(t,l) = 0, & & \frac{\partial^2 w}{\partial x^2}(t,l) = 0,
    \end{align}
\end{subequations}
which stands for a hinged beam.
This ensures both conditions in \eqref{eq:conds_on_BCs} simultaneously.
Note that if the affine functions are not automatically excluded by the boundary condition but the first term in \eqref{eq:SAcond_2} vanishes, then one could quotient $X$ out by the affine functions to ensure strict inequality holds in \eqref{eq:SAcond_2} and continue the analysis in the exact same way as will be done below.

The undamped, unforced, linear beam equation is simply
\begin{displaymath}
    \frac{d^2w}{dt^2}  + Aw = 0,
\end{displaymath}
a conservative system that generates a strongly continuous group on $Y = \dom(A^\frac12) \times X$, as mentioned in Section~\ref{sect:setup}.
The identification of  $A^{\frac12}$ and $\dom(A^\frac12)$ is a nontrivial task, one that depends in a peculiar manner on the exact boundary conditions imposed in $F$.
For fourth-order differential operators, such as the elastic operator \eqref{eq:elastic_A}, this was carried out in \cite{russell1988positive}.

Under the assumption 
\begin{equation}
    \left[ w \frac{\partial}{\partial x} \left( EI \frac{\partial^2w}{\partial x^2} \right) -  \frac{\partial w }{\partial x} EI \frac{\partial^2w}{\partial x^2}  \right]^l_0=0, \qquad  w \in \dom(A),  \tag{A.8}
    \label{eq:ass_dom12}
\end{equation}
we have 
\begin{equation}
    \dom(A^\frac12) = H^{2}([0,l]) \cap \left( \bigcap_{\eta \in G} \ker(\eta) \right),
    \label{eq:domA12}
\end{equation}
where $G := F \cap H^{-2}([0,l])$.
For a proof of this assertion, see Appendix~\ref{appendix:domA12}.

\subsection{The addition of internal damping to the model -- construction of the analytic semigroup} 
\label{sect:beam_damping}

There is no broad agreement in the scientific community on how internal damping mechanisms in mechanical systems should be modeled. 
For an in-depth review specifically in the context of beams, we refer the reader to \cite{russell1992mathematical}.
In the simplest cases, one insists that the linear model should satisfy
\begin{equation}
    \frac{d^2w}{dt^2} + 2 \mu A^\alpha \frac{d w}{dt} + Aw = 0
    \label{eq:Aalphabeammodel}
\end{equation}
for some $\alpha \in [0,1]$ and $0 < \mu \ll 1$.
Here, $\alpha = 0$ is the simplistic model usually associated with external damping (Figure~\ref{fig:1b}),
$\alpha = 1$ stands for the Kelvin-Voigt model (Figure~\ref{fig:1d}), whereas $\alpha= \frac12$ corresponds to the scenario of structural damping, which postulates a linear relationship between frequency and damping (Figure~\ref{fig:1c}).
The cases $\alpha < \frac12$ are not permitted in our setting; in fact, it is shown in \cite[Section~2]{chen1989proof} that in these cases the generated semigroup is not analytic.

\begin{figure*}
    \centering
    \begin{subfigure}[b]{0.495\textwidth}
        \centering
{\phantomsubcaption\label{fig:1a}}
  \tikz\node[inner sep=0pt,label={[anchor=north west]north west:\subref{fig:1a}}] {\includegraphics[width=\textwidth]{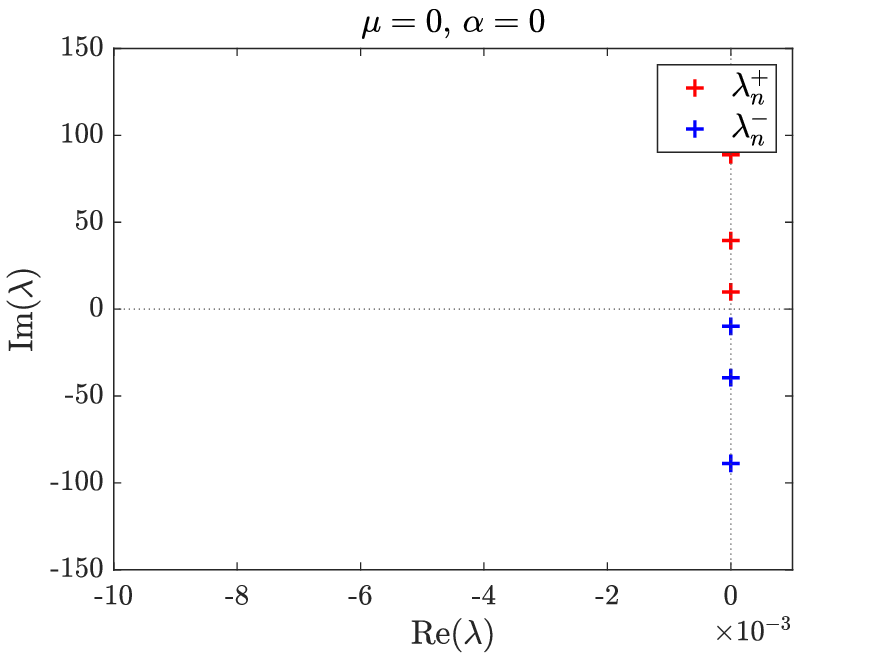}};
    \end{subfigure}
    \hfill
    \begin{subfigure}[b]{0.495\textwidth}  
        \centering 
{\phantomsubcaption\label{fig:1b}}
  \tikz\node[inner sep=0pt,label={[anchor=north west]north west:\subref{fig:1b}}] {\includegraphics[width=\textwidth]{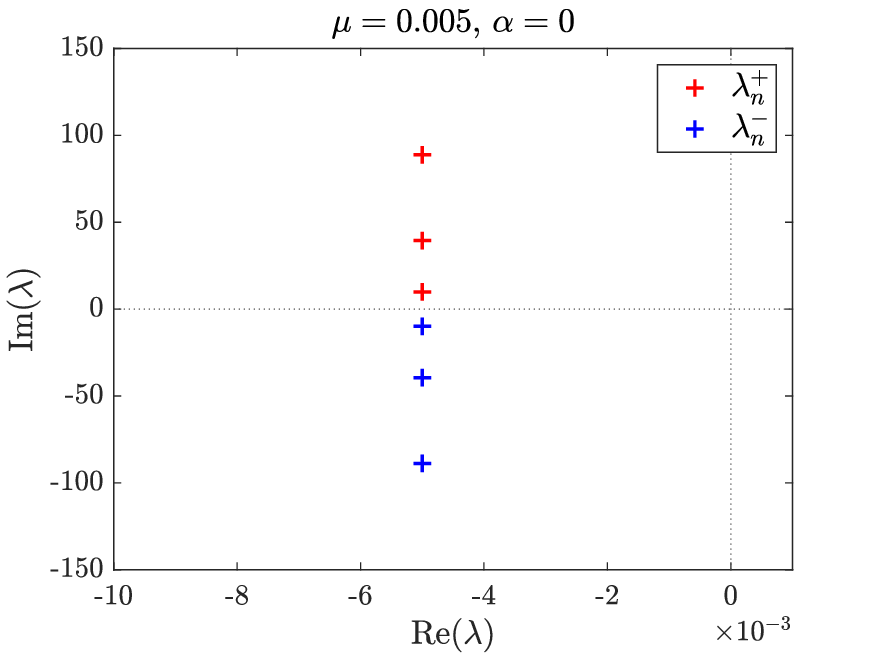}};
    \end{subfigure}
    \vskip\baselineskip
    \begin{subfigure}[b]{0.495\textwidth}  
    \centering
{\phantomsubcaption\label{fig:1c}}
  \tikz\node[inner sep=0pt,label={[anchor=north west]north west:\subref{fig:1c}}] {\includegraphics[width=\textwidth]{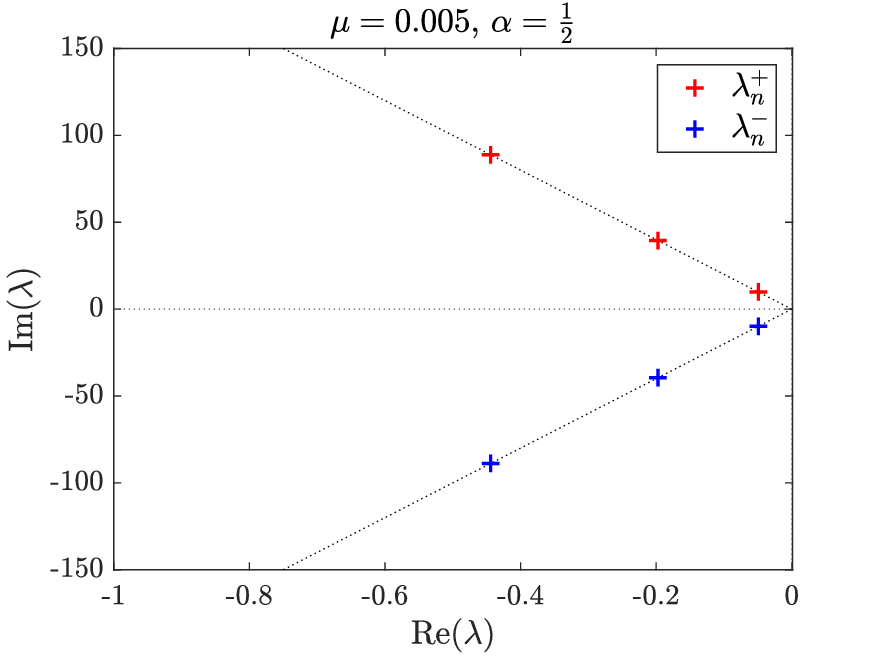}};
    \end{subfigure}
    \hfill
    \begin{subfigure}[b]{0.495\textwidth}   
        \centering 
{\phantomsubcaption\label{fig:1d}}
  \tikz\node[inner sep=0pt,label={[anchor=north west]north west:\subref{fig:1d}}] {\includegraphics[width=\textwidth]{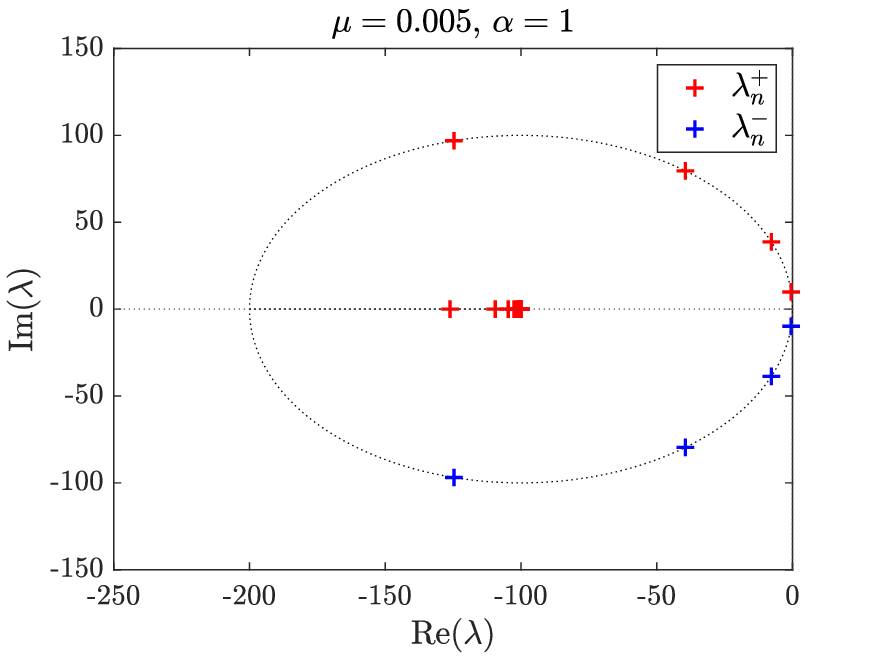}};
    \end{subfigure}
    \caption{Side-by-side comparison of the dominant part of spectra for different values of $\mu$ and $\alpha$, computed for a sample set of parameter values and boundary conditions, namely $EI/\rho = 1$ and $l = 1$ with \eqref{eq:BCs_beam}.
    We remark that the spectra appearing in panels \ref{fig:1c}-\ref{fig:1d} could not be treated by the methodology in \cite{kogelbauer2018rigorous}.}
    \label{fig:1}
\end{figure*}

The physically most appealing  of these is the case $\alpha = \frac12$. 
As already mentioned, the identification of $A^\frac12$ is, however, generally nontrivial.
For the operator $A$ described by \eqref{eq:elastic_A}-\eqref{eq:SAcond},
Theorem~3 of \cite{russell1988positive} asserts that there exists a bounded, invertible operator $P$ on $L^2([0,l])$ such that $PA^\frac12w = -\frac{EI}{\rho}\partial_{xx}w$ for all $w \in \dom(A)$ (see Theorem~\ref{thm:P}).
Under assumption \eqref{eq:ass_dom12}, $P$ is an isometry, and, in general, $P^2 = P$ holds.
If one further assumes that the boundary conditions are such that the eigenvectors of $A$ contain no exponential terms and consist merely of trigonometric functions, then $P = \mathrm{id}_{L^2([0,l])}$ (compare with \cite[equations (1.10)-(1.16)]{russell1988positive}).
The sample boundary conditions listed in \eqref{eq:BCs_beam} fall into this scenario, thus the structural damping scenario in this case yields the linear model
\begin{equation}
     \frac{\partial^2w}{ \partial t^2} -\frac{2 \mu EI}{\rho} \frac{\partial^3 w}{\partial x^2 \partial t} + \frac{EI}{\rho}  \frac{\partial^4 w}{\partial x^4} =0.
     \label{eq:beameq_A12_equal_D2}
\end{equation}

A larger class of damping scenarios are permitted in our setting: Namely, $2\mu A^\alpha$ in \eqref{eq:Aalphabeammodel} could be replaced by an operator $B$ satisfying \ref{A2}.
In \cite{russell1992mathematical}, the damping operator 
\begin{equation}
    [Bw](x) = -2 \mu \frac{\partial}{\partial x} \int_0^l g(x-\xi) [w'(x) - w'(\xi)] d\xi
    \label{eq:B_russell}
\end{equation}
is suggested as a modification to classical structural damping, in order to recover physical meaning behind the  term whenever $P \neq \mathrm{id}_{L^2([0,l])}$.
Under reasonable, physically sound assumptions on $g$ (symmetry and sufficient smoothness) and the boundary conditions in $G$, the damping operator \eqref{eq:B_russell} satisfies assumption \ref{A2} and hence fits in our setting.

\subsubsection{Construction of the analytic semigroup}

As in Section~\ref{sect:setup}, we set the phase space $Y$ to be $\dom(A^\frac12) \times X$ equipped with the energy norm \eqref{eq:norm_E}.
The operator 
\begin{equation}
    \mathcal{L}_B = \begin{pmatrix}
        0 & \mathrm{id}_X \\
        -A & -B
    \end{pmatrix}
    \label{eq:L_B}
\end{equation}
is a sectorial operator that generates an analytic semigroup $t \mapsto e^{\mathcal{L}_B t}$ on $Y^\beta$ for all $\beta \in [0,1)$.
As mentioned earlier, $\mathcal{L}_B$ has compact resolvent when $B$ satisfies \ref{A2} with $\alpha \in [\frac12,1)$ -- we thus have no issue defining spectral projections as we choose.

Recall from Section~\ref{sect:spectrum} (see also Figure~\ref{fig:1}) that eigenvalues $\{ \lambda^\pm_n\}_{n \in \mathbb{N}}$ of $\mathcal{L}_{2\mu A^\alpha}$ for $\alpha \geq 0$ are the solutions of
\begin{displaymath}
    \lambda^2 + 2\mu \sigma_n^\alpha \lambda + \sigma_n = 0, \qquad \text{that is,} \qquad \lambda^\pm_n = \left(-\mu \pm \sqrt{\mu^2-\sigma_n^{1-2\alpha}} \right) \sigma_n^\alpha,
\end{displaymath}
where $\{\sigma_n\}_{n \in \mathbb{N}} \subset \mathbb{R}$ are the eigenvalues of $A$ ordered in ascending order.

For $\alpha > \frac12$, the limiting eigenvalues of $\mathcal{L}_{2 \mu A^\alpha}$ (as $n \to \infty$) turn real.
These modes are usually termed \textit{overdamped modes}.
For $\alpha \in (\frac12,1)$, these overdamped mode-pairs $\lambda_n^\pm$ tend to $-\infty$. 
In the Kelvin-Voigt limit ($\alpha = 1$), the eigenvalue sequence $\lambda_n^+$ tends to $-\frac{1}{2\mu}$ ($\lambda_n^- \to - \infty$ still), hence $-\frac{1}{2\mu} \in \sigma_c(\mathcal{L}_{2\mu A})$ and $\mathcal{L}_{2\mu A}$ does not have compact resolvent.
This is not strictly necessary to apply Theorem~\ref{thm:main}, however, so long as we can define the desired projections.
For the case of $\mathcal{L}_{2\mu A}$, the projections associated to a spectral subset $\Sigma$ of the form \ref{A4} can still be defined for any $\gamma < \sup \mathrm{Re} \, \sigma(\mathcal{L}_{2\mu A})$, $\gamma \neq -\frac{1}{2\mu}$, but for $\gamma < -\frac1{2\mu}$ the spectral subspace $\mathrm{im}(P_\Sigma)$ is infinite dimensional, hence so is the manifold $W^\Sigma$ obtained through Theorem~\ref{thm:main}.
This, therefore, does not result in a reduced-dimensional model for the full system.

\subsection{Permitted class of nonlinearities}
\label{sect:beam_nonlin}

Identification of the domain $\dom(\mathcal{L}_B)$ of \eqref{eq:L_B} and its fractional domains $Y^\beta$ is crucial to understanding what the permissible nonlinearities are, as dictated by our assumption \ref{A3x}.
Recall from Section~\ref{sect:setup} that, in general, we have $\dom(\mathcal{L}_B) \supset \dom(A) \times \dom(B)$ and $\dom(\mathcal{L}_{2\mu A^\alpha}) = \dom(A) \times \dom(A^\alpha)$ for $\alpha \in [\frac12,1]$.
In particular, for the structurally damped model with assumption \eqref{eq:ass_dom12} in place, 
\begin{displaymath}
    \dom(\mathcal{L}_{2\mu A^\frac12}) = H^4([0,l]) \cap \left( \bigcap_{\eta \in F} \ker(\eta) \right) \times H^2([0,l]) \cap \left( \bigcap_{\eta \in G} \ker(\eta) \right).
\end{displaymath}

An application of Lemma~\ref{lemma:embedding} from Appendix~\ref{appendix:domLB} with $k = 2$  yields that $Y^\beta$  associated to $\mathcal{L}_{2\mu A^\frac12}$ is continuously embedded in  $H^3([0,l]) \times H^1([0,l])$ for $\beta > \frac12$.
Moreover, $Y^\beta$ is also continuously embedded in $Y$.
Therefore, it suffices to show that the nonlinearity is of class $C^r$ from $H^3([0,l]) \times H^1([0,l]) \cap Y$ to $Y$, which then implies \ref{A3x}. 
Given the form of $\mathcal{F} = (0,f)$ in \eqref{eq:first_order_1}-\eqref{eq:first_order_2}, we have $\mathrm{im}(\mathcal{F}) \subset \{0\} \times X \subset Y$, thus we obtain that a sufficient condition for \ref{A3x} is that $f$ is of class $C^r$ from  $H^3([0,l]) \times H^1([0,l]) \cap Y$ to $X = L^2([0,l])$.
Hence, $f = f(w,\dot{w})$  can admit three spatial derivatives on $w$ and one on $\dot{w}$.

For instance, the model
\begin{displaymath}
    \frac{\partial^2w}{ \partial t^2} -\frac{2 \mu EI}{\rho} \frac{\partial^3 w}{\partial x^2 \partial t} + \frac{EI}{\rho}  \frac{\partial^4 w}{\partial x^4}  = \frac{\partial^3 w}{\partial x^3} \left( \frac{\partial^2 w}{\partial x \partial t} \right)^3 + w^2\frac{\partial w}{\partial x} 
\end{displaymath}
with boundary conditions \eqref{eq:BCs_beam}
hence satisfies all assumptions of Section~\ref{sect:setup} and admits SSMs for any subdivision of the spectrum of the form \ref{A4} with $\gamma < \sup \mathrm{Re} \, \sigma(\mathcal{L}_{2\mu A^\frac12})$.

\subsubsection{Nonlinear axially restrained elastic beam}
\label{sect:beam_final_model}

We proceed to describe a more physically sound geometric nonlinearity taken from \cite[page 312]{lacarbonara2013nonlinear}, originally derived in \cite{mettler1962dynamic}.
The equation of motion is 
\begin{equation}
      \frac{\partial^2 w}{\partial t^2} - \frac{2 \mu EI}{\rho} \frac{\partial^3 w}{\partial x^2 \partial t}
+ \frac{EI}{\rho} \frac{\partial^4 w}{\partial x^4}
= \left(  \frac{E }{2 \rho l} \int_0^l \left( \frac{\partial w}{\partial x} \right)^2 dx \right)
\frac{\partial^2 w}{\partial x^2}
+ \varepsilon h(t,x),
\label{eq:nonlin_eq}
\end{equation}
with boundary conditions \eqref{eq:BCs_beam}.
Here, all parameters are still as introduced in the introduction to Section~\ref{sect:beam} and $h: \mathbb{R} \times [0,l] \to \mathbb{R}$ is an external forcing term with amplitude $\varepsilon \geq 0$. 

In \cite{lacarbonara2013nonlinear}, \eqref{eq:nonlin_eq} is derived from the theory of unshearable beams under the assumption of no axial forcing, small rotational deformation of the cross sections and negligible rotary inertia.
This model has a limited range of validity, however, it is expected to perform reasonably well in linearly elastic beams that are axially restrained at the boundaries and are subject to moderate loading that is not in resonance with any axial vibration modes, as per \cite[page 312]{lacarbonara2013nonlinear}.
Equation \eqref{eq:nonlin_eq} is studied extensively in \cite{formica2013coupling}.

The nonlinearity in \eqref{eq:nonlin_eq} corresponds to 
\begin{equation}
    f(w,\dot{w}) = \frac{E }{2 \rho l} \Vert w \Vert_{H^1([0,l])}^2 \frac{\partial^2 w}{\partial x^2}
    \label{eq:f_beam}
\end{equation}
in the notation of \eqref{eq:second_order}.
In particular, $f  : H^2([0,l]) \times L^2([0,l]) \to L^2([0,l])$ is of class $C^\infty$. 
By the previous section, this implies that the assumption on the nonlinearity \ref{A3x} is met with $r = \infty$.

\subsection{Computation of the SSM \texorpdfstring{$W^\Sigma$}{}}
\label{sect:beam_mfd}

We shall take as our model \eqref{eq:nonlin_eq}, but non-dimensionalize it as in \cite{formica2013coupling} for reasons of comparability and convenience.
The resulting equation reads
\begin{equation}
          \frac{\partial^2 w}{\partial t^2} - 2 \mu \frac{\partial^3 w}{\partial x^2 \partial t}
+ \frac{\partial^4 w}{\partial x^4}
= \left(  \frac{a}{2 } \int_0^1 \left( \frac{\partial w}{\partial x} \right)^2 dx \right)
\frac{\partial^2 w}{\partial x^2} + \varepsilon h(t,x),
\label{eq:nonlin_eq_nondim}
\end{equation}
where $\mu$ is a rescaled damping parameter and $a$ is a dimensionless coefficient.
For now, we also take $\varepsilon=0$.
The length of the beam has been rescaled to $1$, hence $X =L^2([0,1])$ with the usual $L^2$ inner product.

From here onwards, we denote by $A$ the fourth-order derivative operator.
Equation \eqref{eq:nonlin_eq_nondim} corresponds to
\begin{equation}
    \dot{\xi} = \mathcal{L}_{2 \mu A^\frac12} \xi + \mathcal{F}(\xi)
    \label{eq:beam_first_order}
\end{equation}
in first-order form, where 
\begin{equation}
    \mathcal{F}: Y^\beta \ni  (u,v)\mapsto \left(0,\frac{a}{2 } \int_0^1 (\partial_x u)^2 dx \partial_{xx}  u \right) \in Y.
    \label{eq:mathcalF_beam}
\end{equation}

The spectral quantities of $\mathcal{L}_{2 \mu A^\frac12}$ from  Section~\ref{sect:spectrum} that would normally be required for the manifold calculations are listed in Appendix~\ref{appendix:biorthogonal}. 
Before embarking on such calculations, we notice the following peculiar property of the nonlinearity \eqref{eq:mathcalF_beam}:
\begin{displaymath}
    \mathcal{F} (\psi_n^{+}) = \mathcal{F} (\psi_n^{-}) = \frac{ ia }{8(n\pi)^2 \sqrt{1-\mu^2}} \left(\psi_n^+ - \psi_n^- \right).
\end{displaymath}
From this, we infer that the spectral subspace $Y^\beta_\Sigma$, $\Sigma = \{\lambda_1^+,\lambda_1^-\}$, is invariant under the nonlinear dynamics; we may hence take $W^\Sigma = Y^\beta_\Sigma$.

Parameterizing $Y^\beta_\Sigma \cong \mathbb{C}$ via $z \mapsto z \psi_1^+ + \overline{z} \psi_1^-$, we obtain
\begin{equation}
    \dot{z} = \lambda_1^+ z + \frac{ ia }{8\pi^2 \sqrt{1-\mu^2}} (z+\overline{z})^3.
    \label{eq:reddyn_beam_b0}
\end{equation}
We perform a normal form transformation on system \eqref{eq:reddyn_beam_b0} to match the form of \eqref{eq:R_form}.
The resulting system is
\begin{equation}
    \dot{q} = \lambda_1^+ q + \frac{ 3ia }{8\pi^2 \sqrt{1-\mu^2}} |q|^2q
    \label{eq:reddyn_nf_3ord}
\end{equation}
with $z = p(q,\overline{q})$,
\begin{equation}
    p(q,\overline{q}) =q+ \frac{ ia }{8\pi^2 \sqrt{1-\mu^2}} \left[ \frac{1}{2 \lambda_1^+}q^3 + \frac{3}{2 \lambda_1^-}|q|^2 \overline{q} + \frac{1}{3\lambda_1^--\lambda_1^+} \overline{q}^3 \right] + o(|q|^3)
    \label{eq:p}
\end{equation}
as $|q| \to 0$.
Thus, we have 
\begin{equation}
    K (q,\overline{q})= p(q,\overline{q}) \psi_1^+ + \overline{p(q,\overline{q})} \psi_1^-.
    \label{eq:K_flat}
\end{equation}
(Note that this is precisely the configuration we would have obtained by naively following Section~\ref{sect:param} as well.)

We compute the backbone curve, as in Section~\ref{sect:backbone}. 
For the frequency, we have 
\begin{displaymath}
    \vartheta(r) = \pi^2\sqrt{1-\mu^2} + \frac{ 3a }{8\pi^2 \sqrt{1-\mu^2}}r^2,
\end{displaymath}
whereas the amplitude we compute according to \eqref{eq:Amp} and \eqref{eq:K_flat} with $K $ truncated at third order; the resulting curve $\mathcal{B}$ is plotted in Figure~\ref{fig:2} for different values of $a$.

To make the problem more interesting, we modify the nonlinearity to
\begin{displaymath}
    \mathcal{G}: Y^\beta \ni  (u,v)\mapsto \mathcal{F}(u,v) + b(0,u\partial_{x}u\partial_{x}v) \in Y,
\end{displaymath}
for some $b\in \mathbb{R}$,
so that the governing second-order equation becomes
\begin{equation}
             \frac{\partial^2 w}{\partial t^2} - 2 \mu \frac{\partial^3 w}{\partial x^2 \partial t}
+ \frac{\partial^4 w}{\partial x^4}
= \left(  \frac{a}{2 } \int_0^1 \left( \frac{\partial w}{\partial x} \right)^2 dx \right)
\frac{\partial^2 w}{\partial x^2} + b w \frac{\partial w}{\partial x} \frac{\partial^2 w}{\partial x \partial t}.
\label{eq:mech_eq_a_b}
\end{equation}

We proceed with the computations along the guidelines outlined in Section~\ref{sect:param}.
Inspection of the eigenvalues in Appendix~\ref{appendix:biorthogonal} shows that Theorem~\ref{thm:main} holds with $\ell = 3$.
The nonlinearity $\mathcal{G}$  is of order three and hence will only enter in \eqref{eq:invariance_ordered} for the third-order coefficients.
For convenience, let us denote by $\widehat{\mathcal{G}} :Y^\beta \times Y^\beta \times Y^\beta \to Y$ the map
\begin{displaymath}
    \widehat{\mathcal{G}}(\xi,\zeta,\varrho) := \begin{pmatrix}
        0 \\ \frac{a}{2 } \int_0^1 \partial_x \xi_1 \partial_x \zeta_1 dx \partial_{xx}  \varrho_1 +b\xi_1 \partial_x \zeta_1 \partial_{x}  \varrho_2
    \end{pmatrix}
\end{displaymath}
(here $(\cdot)_j$ takes the $j$-th component, $j =1,2$),
so that $\mathcal{G}(\xi) = \widehat{\mathcal{G}}(\xi,\xi,\xi)$.
Note that $D^3 \mathcal{G}(0) = 6 \widehat{\mathcal{G}}$.

The problem at first order, \eqref{eq:1storder}, is unchanged and hence we take as solution $K_{(1,0)} = \psi_1^+$, $K_{(0,1)} = \psi_1^-$.
Since $\mathcal{F}$ does not have second-order terms, the right-hand sides of \eqref{eq:2ndorder} are $0$.
Thus, the unique solution at second order is given by $K_{(2,0)} = K_{(1,1)} = K_{(0,2)} = 0$.

At third order, matching the $z_1^3$ terms of \eqref{eq:invariance_ordered} yields
\begin{equation}
    \big[\mathcal{L}_{2 \mu A^\frac12} - 3 \lambda_1^+ \big]K_{(3,0)}  = - \frac16 D^3 \mathcal{G}(0) [\psi_1^+,\psi_1^+,\psi_1^+] = -\mathcal{G}(\psi_1^+).
    \label{eq:order_30}
\end{equation}
Proceeding as in Section~\ref{sect:param}, we may project \eqref{eq:order_30} to obtain
\begin{align*}
    \langle K_{(3,0)} , \psi^{*-}_m \rangle_Y  &= \frac{1}{3 \lambda_1^+-\lambda_m^+} \left\langle  \mathcal{G}(\psi_1^+), \psi^{*-}_m\right\rangle_Y, \\
    \langle K_{(3,0)} , \psi^{*+}_m \rangle_Y  &= \frac{1}{3 \lambda_1^+-\lambda_m^-} \left\langle  \mathcal{G}(\psi_1^+), \psi^{*+}_m\right\rangle_Y,
\end{align*}

A subsequent matching of the $z_1^2z_2$ terms of \eqref{eq:invariance_ordered} yields (upon noting $ \mathcal{G}(\psi_1^+) =\widehat{\mathcal{G}}(\psi_1^-,\psi_1^+,\psi_1^+) = \widehat{\mathcal{G}}(\psi_1^+,\psi_1^-,\psi_1^+)$)
\begin{displaymath}
    \big[\mathcal{L}_{2 \mu A^\frac12} - (2 \lambda_1^+ +\lambda_1^-) \big] K_{(2,1)}  = -2 \mathcal{G}(\psi_1^+) - \widehat{\mathcal{G}}(\psi_1^+,\psi_1^+,\psi_1^-)+ R_0 \psi_1^+ 
\end{displaymath}
which in turn yields
\begin{subequations}\label{eq:K21s}
    \begin{align}
    \langle K_{(2,1)} , \psi^{*-}_m \rangle_Y  &= \frac{1}{2 \lambda_1^+ +\lambda_1^--\lambda_m^+} \left\langle   2 \mathcal{G}(\psi_1^+)+ \widehat{\mathcal{G}}(\psi_1^+,\psi_1^+,\psi_1^-) - R_0 \psi_1^+, \psi^{*-}_m\right\rangle_Y, \\
    \langle K_{(2,1)} , \psi^{*+}_m \rangle_Y  &= \frac{1}{2 \lambda_1^+ +\lambda_1^--\lambda_m^-}  \left\langle   2 \mathcal{G}(\psi_1^+) + \widehat{\mathcal{G}}(\psi_1^+,\psi_1^+,\psi_1^-) - R_0 \psi_1^+, \psi^{*+}_m\right\rangle_Y.
\end{align}
\end{subequations}

We proceed to compute the right-hand sides with the quantities listed in Appendix~\ref{appendix:biorthogonal}. 
We have $ \widehat{\mathcal{G}}(\psi_1^-,\psi_1^+,\psi_1^+) = \widehat{\mathcal{G}}(\psi_1^+,\psi_1^-,\psi_1^+) = \mathcal{G}(\psi_1^+)$,
\begin{align*}
    &\mathcal{G}(\psi_1^+) = \mathcal{F}(\psi_1^+) + \begin{pmatrix}
        0 \\
        1 
    \end{pmatrix} \left( -\mu + i \sqrt{1-\mu^2}\right) \frac{b}{4\pi^2} \left( \sin(\pi x) + \sin (3 \pi x) \right), \\
    &\widehat{\mathcal{G}}(\psi_1^+,\psi_1^+,\psi_1^-) = \mathcal{F}(\psi_1^+) + \begin{pmatrix}
        0 \\
        1 
    \end{pmatrix} \left( -\mu - i \sqrt{1-\mu^2}\right) \frac{b}{4\pi^2} \left( \sin(\pi x) + \sin (3 \pi x) \right),
\end{align*}
and in turn
\begin{align*} 
    &\mathcal{G}(\psi_1^+) = c_1^+ \left(\psi_1^+ - \psi_1^- \right) + c_3^+(\psi_3^+-\psi_3^-), \\
    &\widehat{\mathcal{G}}(\psi_1^+,\psi_1^+,\psi_1^-) = c_1^- \left(\psi_1^+ - \psi_1^- \right) + c_3^-(\psi_3^+-\psi_3^-),
\end{align*}
where
\begin{displaymath}
    c_1^\pm := \frac{ i(a +b\mu)}{8\pi^2 \sqrt{1-\mu^2}} \pm \frac{b}{8\pi^2}  , \qquad  c_3^\pm:= \frac{b}{8\pi^2} \left(\frac{i \mu}{\sqrt{1-\mu^2}}\pm1 \right).
\end{displaymath}
We therefore obtain
\begin{align*}
    \langle K_{(3,0)} , \psi^{*-}_m \rangle_Y &=  \frac{c_1^+}{2\lambda_1^+} \delta_{1m} +\frac{c_3^+}{3\lambda_1^+ - \lambda_3^+} \delta_{3m}, \\
    \langle K_{(3,0)} , \psi^{*+}_m \rangle_Y &=  -\frac{c_1^+}{3\lambda_1^+-\lambda_1^-} \delta_{1m} -\frac{c_3^+}{3\lambda_1^+ - \lambda_3^-} \delta_{3m}.
\end{align*}

As discussed in Section~\ref{sect:param}, $R_0$ is chosen to eliminate the near resonant term; here this amounts to letting 
\begin{displaymath}
    R_0 = 2c_1^+ + c_1^-.
\end{displaymath}
Plugging this back into the equations \eqref{eq:K21s}, we also obtain
\begin{align*}
    \langle K_{(2,1)} , \psi^{*-}_m \rangle_Y  &= \frac{2c_3^++c_3^-}{2 \lambda_1^+ +\lambda_1^--\lambda_3^+} \delta_{3m},  \\
    \langle K_{(2,1)} , \psi^{*+}_m \rangle_Y  &=-\frac{2c_1^++c_1^-}{2 \lambda_1^+} \delta_{1m}  - \frac{2c_3^++c_3^-}{2 \lambda_1^+ +\lambda_1^--\lambda_3^-} \delta_{3m} .
\end{align*}

Overall, we have (upon precomposing $K=\hat{K}|_{\Delta_c}$ with the map $z \mapsto (z,\overline{z})$),
\begin{align*}
    K(z) &=  K^0(z) + \overline{K^0(z)}+o(|z|^3), \qquad \text{as } |z| \to 0\\
    \dot{z} &= \lambda_1^+ z  +(2c_1^+ + c_1^-)|z|^2z,
\end{align*}
with
\begin{align*}
    K^0(z)& = \left(z + \frac{c_1^+}{2\lambda_1^+} z^3  \right)\psi_1^+   +\left(-\frac{c_1^+}{3\lambda_1^+-\lambda_1^-}z^3-\frac{2c_1^++c_1^-}{2 \lambda_1^+}|z|^2z\right)\psi_1^-\\  
    &+ \left(\frac{c_3^+}{3\lambda_1^+ - \lambda_3^+}z^3 +  \frac{2c_3^++c_3^-}{2 \lambda_1^+ +\lambda_1^--\lambda_3^+}|z|^2z \right) \psi_3^+, \\
    &+ \left(-\frac{c_3^+}{3\lambda_1^+ - \lambda_3^-} z^3  - \frac{2c_3^++c_3^-}{2 \lambda_1^+ +\lambda_1^--\lambda_3^-} |z|^2z\right) \psi_3^-,
\end{align*}
or in polar coordinates (pulling back over $(r,\theta) \mapsto r e^{i\theta} = z$),
\begin{subequations} \label{eq:K_and_reddyn_beam}
\begin{align}
    K(r,\theta) &= 2 \mathrm{Re} \,K^0(re^{i \theta}) + o(r^3), \qquad \text{as } r \to 0, \\
    \dot{r} &=  -\mu \pi^2 r + \frac{b}{8\pi^2}r^3, \\
    \dot{\theta} &=\pi^2 \sqrt{1-\mu^2}+ \frac{3(a+b\mu)}{8\pi^2 \sqrt{1-\mu^2}}r^2 = \vartheta(r).
\end{align}
\end{subequations}
Using this information, we may now compute the backbone curve as in Section~\ref{sect:backbone}.
We plot the resulting curve for a few values of $a,b$ in Figure~\ref{fig:2}.

\begin{figure*}
    \centering
    \begin{subfigure}[b]{0.495\textwidth}
        \centering
        \includegraphics[width=\textwidth]{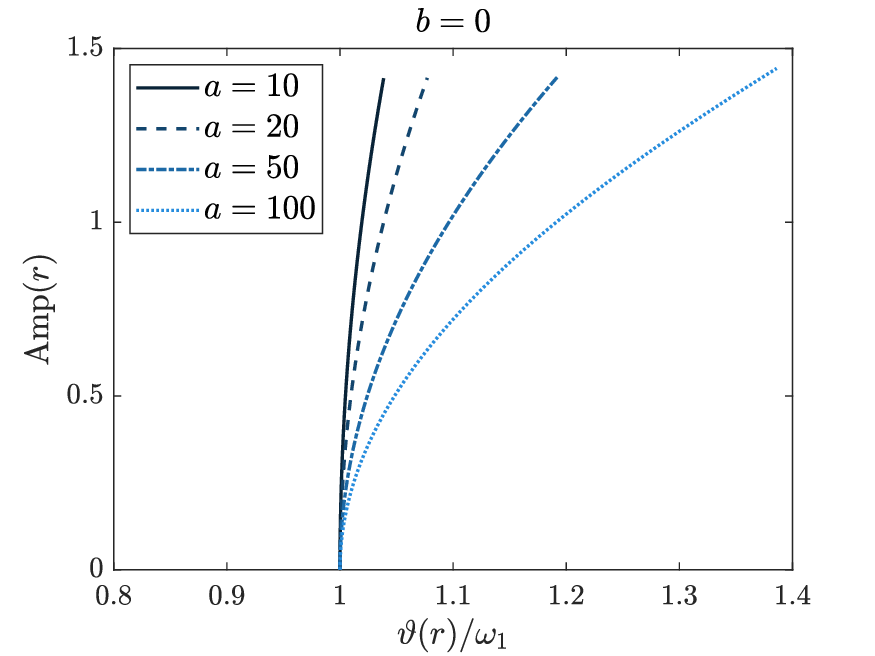}
    \end{subfigure}
    \hfill
    \begin{subfigure}[b]{0.495\textwidth}  
        \centering 
        \includegraphics[width=\textwidth]{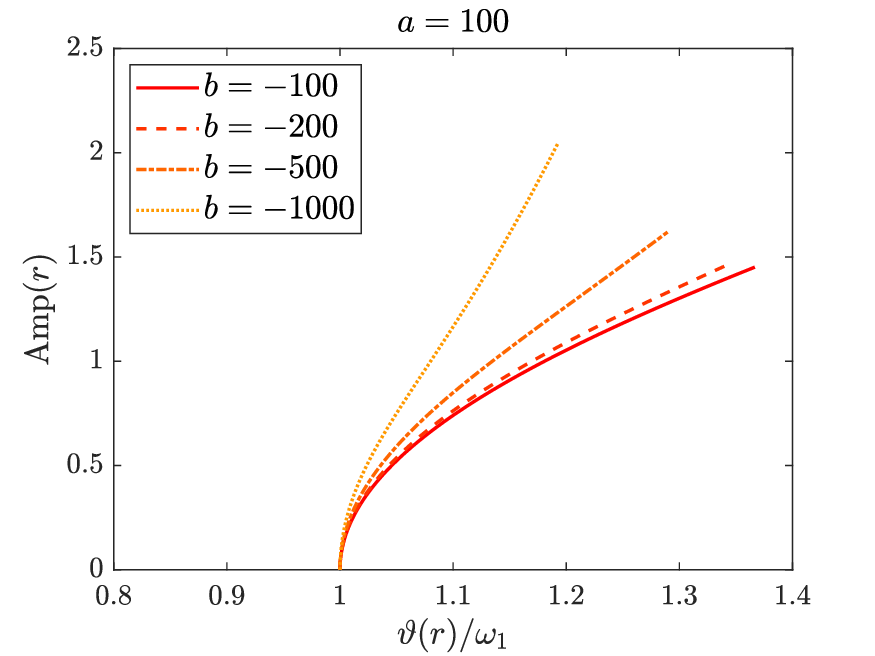}
    \end{subfigure}
    \caption{The images of the backbone curves $\mathcal{B}$ from \eqref{eq:backbone} over $r \in [0,1]$ (third order approximation, according to $K^{\leq 3}$) with varying parameter values $a$ and $b$, for the mechanical system \eqref{eq:mech_eq_a_b}.
    The horizontal, frequency axis is normalized by the first (damped) natural frequency $\omega_1 = \mathrm{Im}(\lambda_1^+)$. 
    The damping parameter is fixed at $\mu = 0.05$.}
    \label{fig:2}
\end{figure*}

In the left panel of Figure~\ref{fig:3}, the evolution of a single trajectory along the third order approximation $\mathrm{im}(K^{\leq 3})$ of the manifold $W^\Sigma$, depicted using the two tangential coordinates $\mathrm{Re} \langle K^{\leq 3}(z), \psi^{*-}_1 \rangle_Y$, $\mathrm{Im} \langle K^{\leq 3}(z), \psi^{*-}_1 \rangle_Y$, and the norm of the complementary components
\begin{equation}
    K^\perp(z) : = K^{\leq 3}(z) - \langle K^{\leq 3}(z), \psi^{*-}_1 \rangle_Y \psi_1^+-\langle K^{\leq 3}(z), \psi^{*+}_1 \rangle_Y \psi_1^-.
    \label{eq:Kperp}
\end{equation}
In the right panel, the space-time evolution of the same sample trajectory is shown.
This trajectory is advected using \eqref{eq:K_and_reddyn_beam} on the reduced space and is subsequently mapped into $Y$ via $K^{\leq 3}$.

\begin{figure*}
    \centering
    \begin{subfigure}[b]{0.495\textwidth}
        \centering
        \includegraphics[width=\textwidth]{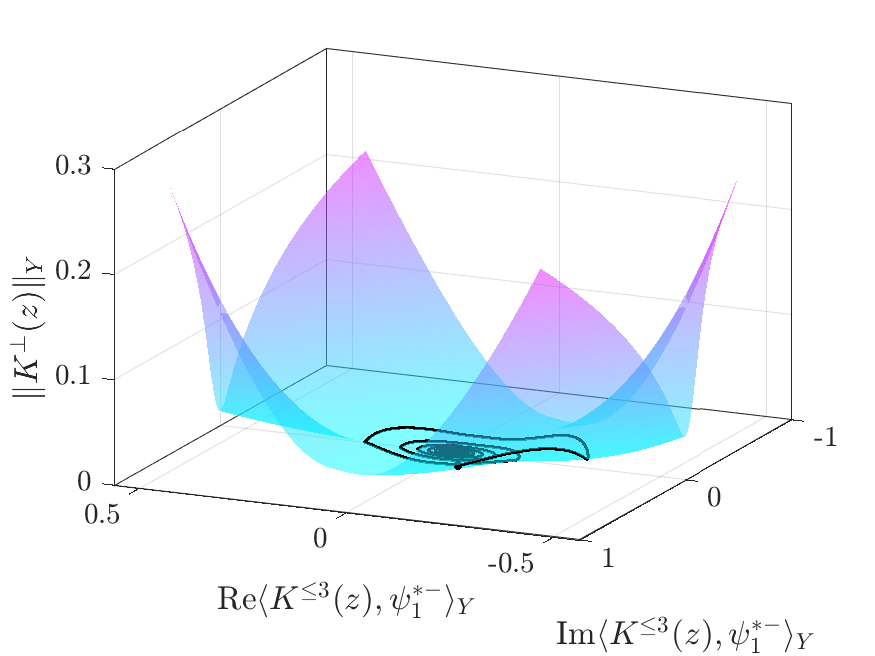}
    \end{subfigure}
    \hfill
    \begin{subfigure}[b]{0.495\textwidth}  
        \centering 
        \includegraphics[width=\textwidth]{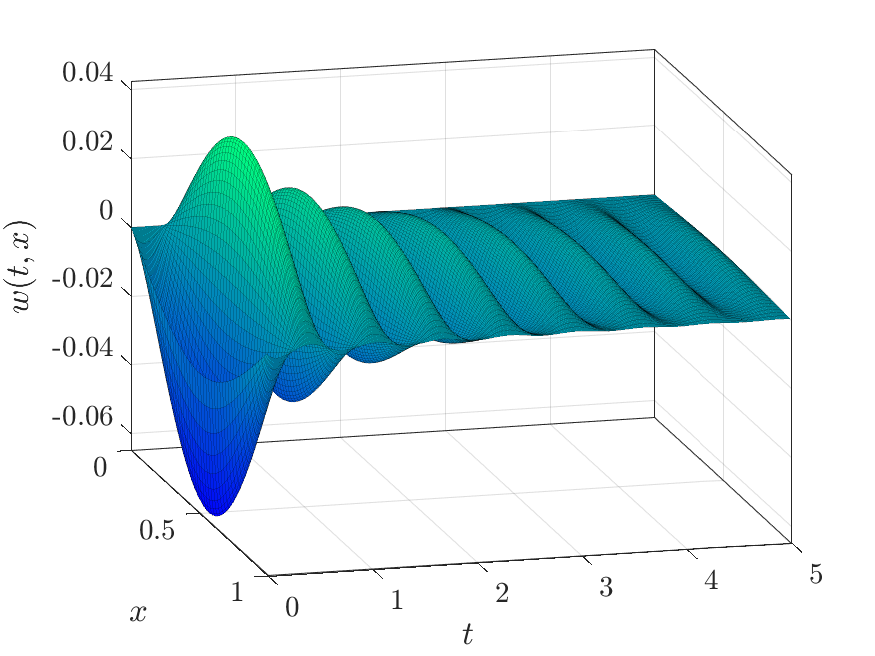}
    \end{subfigure}
    \caption{(Left) The approximate manifold for the mechanical system \eqref{eq:mech_eq_a_b}, embedded in phase space $Y$ with a sample trajectory whose evolution is described by \eqref{eq:K_and_reddyn_beam}, initiated at $(r_0,\theta_0) = (0.4,2)$. (Right) The same sample trajectory shown using the physical coordinates $(t,x,w(t,x))$. 
    The parameters are fixed as $\mu = 0.05$, $a =100$ and $b = -1000$.}
    \label{fig:3}
\end{figure*}

\subsection{The addition of forcing}
\label{sect:beam_forcing}

We now describe briefly how the backbone curve perturbs under weak periodic forcing and how one can extract a frequency response curve from the reduced dynamics.
We take \eqref{eq:nonlin_eq_nondim} (i.e., the $b = 0$ case) as our model and take for the particular forcing
\begin{displaymath}
    h(t,x)  = \sin (\pi x) \cos (\omega t),
\end{displaymath}
for some forcing frequency $\omega > 0$.
This choice of $h$ is meant to replicate the forcing $\lambda \cos (\omega t)$ in \cite{formica2013coupling} and satisfy the boundary conditions \eqref{eq:BCs_beam}.
For the parameters of the model, we  mimic those of \cite{formica2013coupling} by taking $\mu = 0.395/\pi^2$ and $a = 1.2 \cdot 10^4$. Note that only a qualitative comparison will be feasible due to the difference in forcing and damping mechanisms (\cite{formica2013coupling} uses the external damping scenario, i.e., $\alpha = 0$).

Note that $h$ is aligned with the spectral subspace $Y^\beta_\Sigma = W^\Sigma$.
Given the periodicity of $h(\cdot,x)$, we introduce $\phi:= \omega t$ to model the phase space extension on the compact space $\mathbb{S} := \mathbb{R}/2\pi \mathbb{Z}$; so that \eqref{eq:nonlin_eq_nondim} can be modeled as an autonomous system on $Y \times \mathbb{S}$ with the extension $\dot{\phi} = \omega$.
We let $\tilde{h}(\phi,x) := h(\phi/\omega,x)$, and  declare the first-order form of $h$ as $\mathcal{H}: \mathbb{S} \to Y$ given by 
\begin{displaymath}
    \mathcal{H}(\phi) := \begin{pmatrix}
        0 \\ \tilde{h}(\phi,\cdot)
    \end{pmatrix}
     = -\frac{i  \cos(\phi)}{2\sqrt{1-\mu^2}} \left( \psi_1^+ - \psi_1^- \right) .
\end{displaymath}
The addition of $\varepsilon\mathcal{H}$ to \eqref{eq:beam_first_order} results in the reduced dynamics \eqref{eq:reddyn_beam_b0} changing to
\begin{subequations} \label{eq:forcedreddyn}
\begin{align}
    \dot{z} &= \lambda_1^+ z + \frac{ ia }{8\pi^2 \sqrt{1-\mu^2}} (z+\overline{z})^3  - \varepsilon\frac{i  \cos(\phi)}{2\sqrt{1-\mu^2}}, \\
    \dot{\phi} &= \omega,
\end{align}
\end{subequations}
on $ \mathbb{C} \times \mathbb{S} \cong W^\Sigma \times \mathbb{S}$.

To extract the backbone curve, we transform \eqref{eq:forcedreddyn} to be in a normal form, which, having already computed its autonomous third order part in \eqref{eq:reddyn_nf_3ord}, we take to be 
\begin{subequations} \label{eq:forcedreddyn_nf}
\begin{align}
    \dot{q} &= \lambda_1^+ q + \frac{ 3ia }{8\pi^2 \sqrt{1-\mu^2}} |q|^2q  + \varepsilon c_h e^{i \phi}, \\
    \dot{\phi} &= \omega,
\end{align}
\end{subequations}
for some $c_h \in \mathbb{C}$.
In general, a conjugacy mapping between \eqref{eq:forcedreddyn} and \eqref{eq:forcedreddyn_nf} is a function of the form $z=\tilde{p}(q,\overline{q},\phi)$, which we shall approximate via a Taylor expansion.
For our purposes, it suffices to compute this to first order in $\varepsilon$ and to order three in $|z|$ with no mixed terms of the form $|z|^i\varepsilon^j$, as per the recommendations of \cite{szalai2017nonlinear,breunung2018explicit}.
We hence compute 
\begin{displaymath}
    \tilde{p}(q,\overline{q},\phi ) = p(q,\overline{q})  + \varepsilon\Xi (\phi) + o(\varepsilon,|q|^3),
\end{displaymath}
where $p$ is the same as in the autonomous case \eqref{eq:p} and the non-autonomous first order correction $\Xi : \mathbb{S} \to \mathbb{C}$ satisfies the invariance equation
\begin{displaymath}
    (\lambda_1^+ - \omega \partial_\phi) \Xi(\phi) = c_h e^{i \phi} +i\frac{e^{i \phi} + e^{- i \phi}}{4 \sqrt{1-\mu^2}}.
\end{displaymath}
In order to avoid division by $i - \lambda_1^+/\omega \approx 0$, we choose $c_h = -i/(4 \sqrt{1-\mu^2})$ and
\begin{displaymath}
    \Xi (\phi)= \frac{i}{4 \sqrt{1-\mu^2} (i \omega + \lambda_1^+)} e^{-i\phi}.
\end{displaymath}

The system \eqref{eq:forcedreddyn_nf} may be put into polar coordinates as
\begin{subequations} \label{eq:forcedreddyn_nf_polars}
\begin{align}
    \dot{r} &= f_1(r)  + \varepsilon \mathrm{Re } (c_h e^{i \phi}), \\
    \dot{\theta} &= f_2(r) + \frac{\varepsilon}{r}\mathrm{Im } (c_h e^{i \phi}) \\
    \dot{\phi} &= \omega,
\end{align}
\end{subequations}
For systems of the form \eqref{eq:forcedreddyn_nf_polars}, \cite[Theorem~3.8]{breunung2018explicit} shows that the forced response curve is given by zeros of the function
\begin{equation}
    f_{\mathrm{FRC}}(r,\omega)  = [f_1(r)]^2 + (f_2(r)-\omega)^2r^2-\varepsilon^2 (\mathrm{Im}\,c_h)^2.
    \label{eq:FRC}
\end{equation}
Note that if one seeks local maxima of $f_{\mathrm{FRC}}$, one arrives precisely at the autonomous definition of the backbone curve in Section~\ref{sect:backbone}; here, however, the meaning of $r$ had to be re-adjusted based on the transformation $\Xi$.
Forced response curves computed according to \eqref{eq:FRC} are plotted in Figure~\ref{fig:4} in the plane of $w_{max}(1/2) = \sup_t w(t,1/2)$ and $\omega$, where $w$ is computed according to $K = \tilde{p} \psi_1^+ + \overline{\tilde{p}} \psi_1^-$.

\begin{figure}
     \centering
     \includegraphics[width=0.5\textwidth]{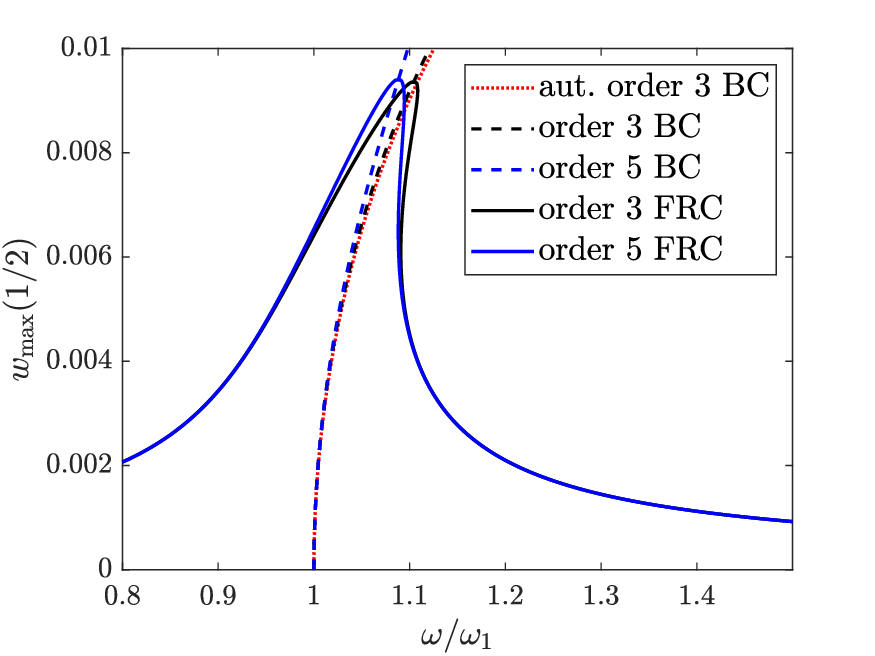}
     \caption{Backbone curves (BCs) and forced response curves (FRCs) for the mechanical system \eqref{eq:nonlin_eq_nondim} computed from eq.~\eqref{eq:FRC}.
     The parameters are fixed as $\mu = 0.395/\pi^2$, $a = 1.2 \cdot 10^5$, $b = 0$ and $\varepsilon=0.075$.}
     \label{fig:4}
\end{figure}

It is sometimes beneficial to increase the order of autonomous normal form by addition of terms of the form $\beta_j|q|^{2j}q$, so as to avoid further near-resonances of the form $(j+1) \lambda_1^+ + j \lambda_1^- \approx \lambda_1^+$ if the underlying structure is damped very lightly ($\mu \ll 1$), see \cite{szalai2017nonlinear}.
In Figure~\ref{fig:4}, we also include results of an order five calculation for the sake of comparison.

Qualitatively, the observations in Figure~\ref{fig:4} support those reported in \cite{formica2013coupling} in terms of the initial hardening behavior.
A direct comparison is not possible due to the differing forcing profiles.

\section{Example: SSM reduction of a nonlinear Kirchhoff-Love plate}
\label{sect:plate}

As a second example, we consider the following nonlinear perturbation of the classical (damped) Kirchhoff-Love plate model on a rectangular domain $\Omega = [0,a] \times [0,b]$,
\begin{equation}
    \ddot{w} - 2\mu \Delta \dot{w} + \Delta^2 w = \kappa w^3,
    \label{eq:plate_model}
\end{equation}
for some $\kappa \in \mathbb{R}$,
subject to the boundary conditions
\begin{equation}
    w|_{\partial \Omega} = 0, \qquad \Delta w |_{\partial \Omega} = 0.
    \label{eq:plate_BCs}
\end{equation}
We shall assume $a/b \in \mathbb{R} \setminus \mathbb{Q}$ so that we do not have to deal with eigenvalues of higher multiplicity.
We take, as usual, $X:= L^2(\Omega)$.

In \eqref{eq:plate_model}, the elastic operator is $A = \Delta^2 : \dom(A) \to X$ defined on the domain
\begin{displaymath}
    \dom(A) = \Big\{ w \in H^4(\Omega) \; \big| \; Tw = 0, \; T \Delta w= 0 \Big\}
\end{displaymath}
encompassing the boundary conditions \eqref{eq:plate_BCs} via the trace operator $T : H^1(\Omega) \to L^2(\partial \Omega)$. 
The operator $A$ is  positive and self-adjoint with eigenvalue-eigenvector pairs of the form
\begin{displaymath}
    \sigma_{nm} = \left[ \left(\frac{n\pi}{a}\right)^2 + \left(\frac{m\pi}{b}\right)^2\right]^2,
\qquad \phi_{nm}(x) = \sqrt{\frac{2}{ab \sigma_{nm}}} \sin\left(\frac{n\pi x}{a}\right)
\sin\left(\frac{m\pi y}{b}\right), 
\end{displaymath}
$n,m \in \mathbb{N}$.
Here $\phi_{nm}$ is normalized such that the corresponding eigenfunction $\psi_{nm}^\pm$ of the operator $\mathcal{L}_{2\mu A^\frac12}$ is of norm one (see Appendix~\ref{appendix:biorthogonal_plate}).

While not explicitly spelled out in \cite{russell1988positive}, it follows from an analogous argument that for the considered set of boundary conditions \eqref{eq:plate_BCs}, we have $A^\frac12 = -\Delta$ on
\begin{displaymath}
    \dom(A^\frac12) = \Big\{ w \in H^2(\Omega) \; \big| \; Tw = 0 \Big\}.
\end{displaymath}
Hence the damping term in \eqref{eq:plate_model} corresponds to the structural damping scenario.

The first order system associated to \eqref{eq:plate_model} and \eqref{eq:plate_BCs} reads
\begin{displaymath}
    \dot{\xi} = \mathcal{L}_{2 \mu A^\frac12} \xi + \mathcal{F}(\xi)
\end{displaymath}
on the phase space $Y = \dom(A^\frac12) \times X$, with 
\begin{displaymath}
    \mathcal{F} : Y \ni (u,v) \mapsto (0, u^3) \in Y.
\end{displaymath}
(The function takes values in $Y$ since $\dom(A^\frac12) \subset L^\infty(\Omega)$ by Sobolev embedding.)
Note that due to the nonlinearity not containing derivatives, we do not need to contend with $Y^\beta$ spaces here.
We still note that for an equation of the form \eqref{eq:plate_model}, three spatial derivatives on $w$ and one on $\dot{w}$ would be permissible, as an application of Lemma~\ref{lemma:embedding} shows.

\subsection{Computation of the SSM \texorpdfstring{$W^\Sigma$}{}}
\label{sect:plate_computation}

We follow once more the procedure outlined in Section~\ref{sect:param}.
For the spectral subset \ref{A4}, we take $\Sigma = \{\lambda_{11}^+,\lambda_{11}^-\}$. 
For the explicit eigenvalues, see Appendix~\ref{appendix:biorthogonal_plate}; with the choice $a = 1$ and $b = \pi/3$ (which we fix throughout), Theorem~\ref{thm:main} holds with $\ell = 2$.
The linear and second order problems obtained from \eqref{eq:invariance_ordered} proceed analogously to the previous example, the solutions are $K_{(1,0)} = \psi_{11}^+$, $K_{(0,1)} = \psi_{11}^-$ (see Appendix~\ref{appendix:biorthogonal_plate}), and $K_{(2,0)} = K_{(1,1)} = K_{(0,2)} = 0$.

For the third order computations, we first remark that $D^3 \mathcal{F}(0)[\psi_{11}^\pm,\psi_{11}^\pm,\psi_{11}^\pm] =6 \mathcal{F}(\psi_{11}^+)$, and that
\begin{displaymath}
    \mathcal{F}(\psi_{11}^+) = c \left[9(\psi_{11}^--\psi_{11}^+) -3(\psi_{31}^--\psi_{31}^+) - 3(\psi_{13}^--\psi_{13}^+) +(\psi_{33}^--\psi_{33}^+) \right],
\end{displaymath}
where 
\begin{displaymath}
    c:=\frac{i \kappa}{16ab \sqrt{1-\mu^2} \sigma_{11}^{3/2}}.
\end{displaymath}
Matching the $z_1^3$ terms of \eqref{eq:invariance_ordered} leads to
\begin{displaymath}
    \left[\mathcal{L}_{2 \mu A^\frac12} - 3 \lambda_{11}^+ \right] K_{(3,0)} =  -\mathcal{F}(\psi_{11}^+).
\end{displaymath}
This implies, via a procedure analogous to Section~\ref{sect:beam_mfd},
\begin{align*}
    &K_{(3,0)} = \frac{9c}{3 \lambda_{11}^+-\lambda_{11}^-}\psi_{11}^- - \frac{9c}{2\lambda_{11}^+}\psi_{11}^+ - \frac{3c}{3\lambda_{11}^+-\lambda_{31}^-} \psi_{31}^- +\frac{3c}{3\lambda_{11}^+-\lambda_{31}^+} \psi_{31}^+ \\
    &- \frac{3c}{3\lambda_{11}^+-\lambda_{13}^-} \psi_{13}^- +\frac{3c}{3\lambda_{11}^+-\lambda_{13}^+} \psi_{13}^+ + \frac{c}{3 \lambda_{11}^+-\lambda_{33}^-} \psi_{33}^- - \frac{c}{3 \lambda_{11}^+-\lambda_{33}^+} \psi_{33}^+.
\end{align*}

Matching the $z_1^2z_2$ terms, we get
\begin{displaymath}
    \big[\mathcal{L}_{2 \mu A^\frac12} - (2 \lambda_{11}^+ +\lambda_{11}^-) \big] K_{(2,1)}  = -3 \mathcal{F}(\psi_{11}^+)+ R_0 \psi_{11}^+,
\end{displaymath}
so that
\begin{displaymath}
    R_0 = -27c
\end{displaymath}
and
\begin{align*}
    &K_{(2,1)} = \frac{27c}{2 \lambda_{11}^+}\psi_{11}^- - \frac{9c}{2 \lambda_{11}^+ +\lambda_{11}^--\lambda_{31}^-} \psi_{31}^- +\frac{9c}{2 \lambda_{11}^+ +\lambda_{11}^--\lambda_{31}^+} \psi_{31}^+ \\
    &- \frac{9c}{2 \lambda_{11}^+ +\lambda_{11}^--\lambda_{13}^-} \psi_{13}^- +\frac{9c}{2 \lambda_{11}^+ +\lambda_{11}^--\lambda_{13}^+} \psi_{13}^+ + \frac{3c}{2 \lambda_{11}^+ +\lambda_{11}^--\lambda_{33}^-} \psi_{33}^- \\
    &- \frac{3c}{2 \lambda_{11}^+ +\lambda_{11}^--\lambda_{33}^+} \psi_{33}^+.
\end{align*}

The full embedding map and reduced dynamics with respect to polar coordinates are
\begin{subequations} \label{eq:K_and_reddyn_plate}
\begin{align}
    K(r,\theta) &= 2r \mathrm{Re} \left[ e^{i \theta} \psi_{11}^+ + r^2 \left( K_{(3,0)} e^{3 i \theta} + K_{(2,1)} e^{ i \theta} \right) \right] + o(r^3), \qquad \text{as } r \to 0, \\
    \dot{r} &=  -\mu \sqrt{\sigma_{11}} r , \\
    \dot{\theta} &=\sqrt{\sigma_{11}(1-\mu^2)}- \frac{27 \kappa}{16ab \sqrt{1-\mu^2} \sigma_{11}^{3/2}}r^2 = \vartheta(r).
\end{align}
\end{subequations}
The backbone curves, as $\kappa$ is varied, are shown in Figure~\ref{fig:5}.
Figure~\ref{fig:6} shows the approximate manifold embedded in $Y$, depicted via coordinates $\mathrm{Re} \langle K^{\leq 3}(z), \psi^{*-}_1 \rangle_Y$, $\mathrm{Im} \langle K^{\leq 3}(z), \psi^{*-}_1 \rangle_Y$ and $\Vert K^\perp (z) \Vert_Y$ (as in \eqref{eq:Kperp}).
On it, a sample trajectory is shown whose evolution is described by \eqref{eq:K_and_reddyn_plate}.
The same trajectory is also shown with respect to physical coordinates in Figure~\ref{fig:7}.

\begin{figure}
     \centering
     \includegraphics[width=0.5\textwidth]{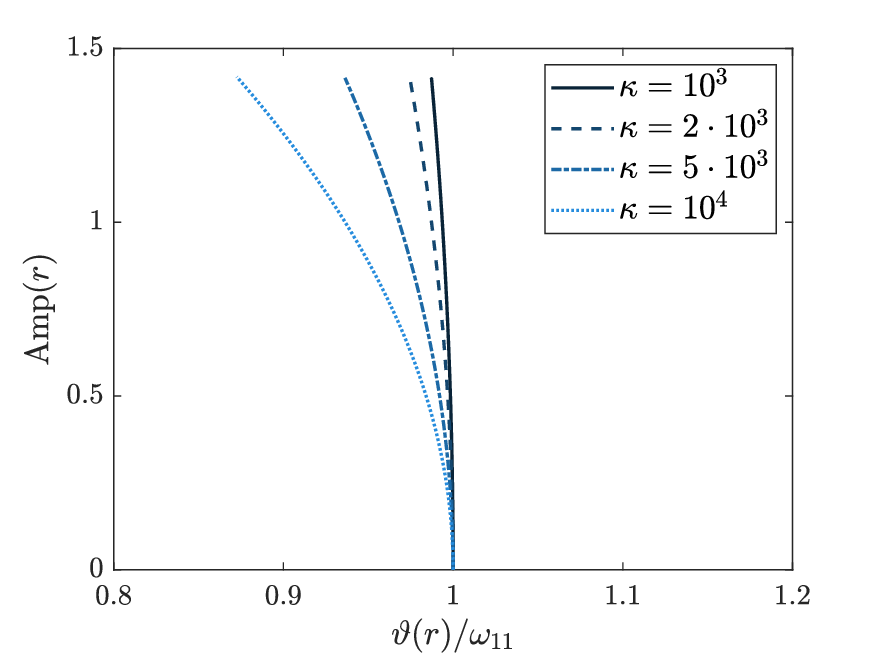}
     \caption{The images of the backbone curves $\mathcal{B}$ from \eqref{eq:backbone} over $r \in [0,1]$ (third order approximation, according to $K^{\leq 3}$) with varying $\kappa$ values, for the mechanical system \eqref{eq:plate_model}.
    The horizontal, frequency axis is normalized by the first (damped) natural frequency $\omega_{11} = \mathrm{Im}(\lambda_{11}^+)$. 
    The rest of the parameters are fixed as $\mu = 0.05$, $a = 1$ and $b = \pi/3$.}
     \label{fig:5}
\end{figure}

\begin{figure}
     \centering
     \includegraphics[width=0.5\textwidth]{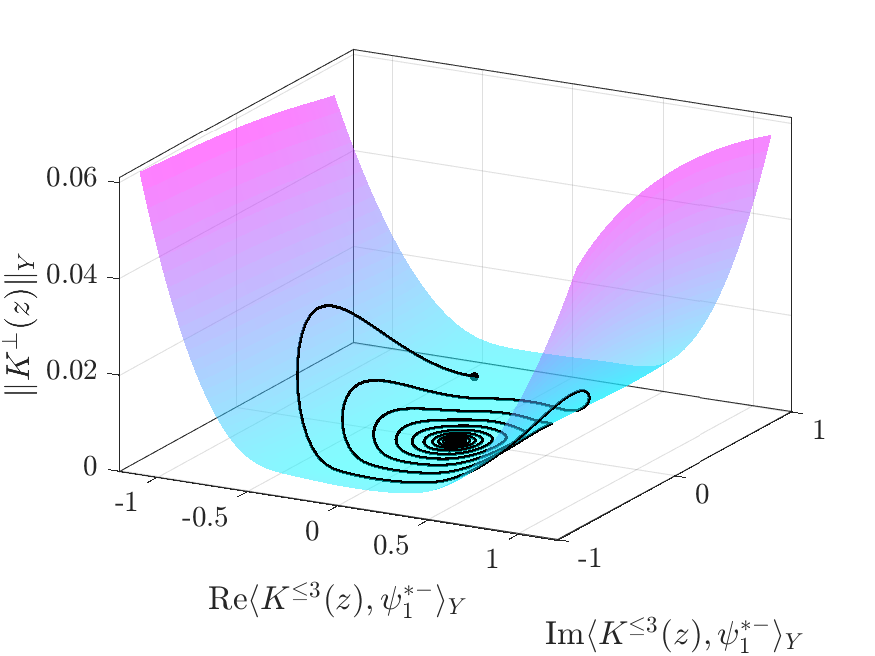}
    \caption{The approximate manifold embedded in phase space $Y$ for the mechanical system \eqref{eq:plate_model} with a sample trajectory whose evolution is described by \eqref{eq:K_and_reddyn_plate}, initiated at $(r_0,\theta_0) = (0.8,2)$.
    The parameters are fixed as $\mu = 0.05$, $\kappa = 10^4$, $a = 1$ and $b = 3/\pi$.}
    \label{fig:6}
\end{figure}

\begin{figure}
     \centering
     \includegraphics[width=1\textwidth]{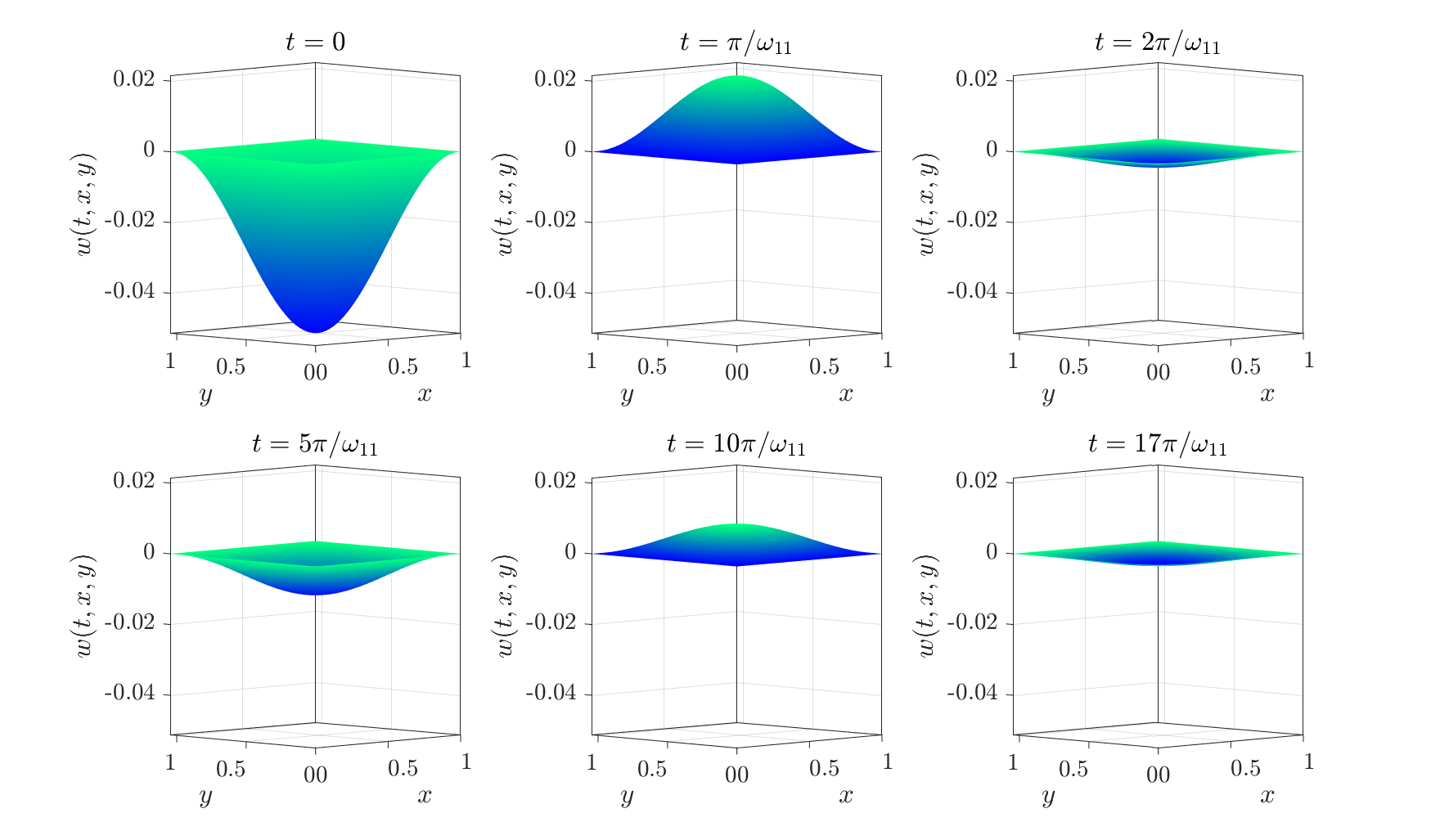}
    \caption{The sample trajectory from Figure~\ref{fig:6} shown using physical coordinates $(t,x,y,w(t,x,y))$. 
    The time instants shown are integer multiples of the half-period $\pi/\omega_{11}$ associated to the (damped) natural frequency $\omega_{11} = \mathrm{Im}(\lambda_{11}^+)$.
    The parameters are fixed as $\mu = 0.05$, $\kappa = 10^4$, $a = 1$ and $b = 3/\pi$.}
    \label{fig:7}
\end{figure}

\section{Conclusions}
\label{sect:conclusions}

In this paper, we have derived existence and regularity results for spectral submanifolds (SSMs) in evolutionary systems generated by PDEs describing nonlinear continuum vibrations \ref{C1}.
This provides previously missing, rigorous mathematical support for SSM-based model reduction procedures commonly employed in finite-element models by confirming their validity in the limiting case governed by the PDE.
Our results enable the notions and physical interpretations drawn from these reduced models the literature has grown to embrace 
to be transferred to the original PDE describing the physical continuum \ref{C2}.

We have also proposed a procedure for hand calculations of SSMs in mechanical systems, generalizing the work of Kogelbauer \& Haller \cite{kogelbauer2018rigorous} through the use of a bi-orthogonal frame spanning the phase space $Y$ \cite[Appendix~A]{chen1989proof}.
While more classical invariant manifolds of mechanical PDEs were already computed via expansions as early as the nineties \cite{roberts1993invariant,shaw1994normal}, the reader may still find some aspects of the systematic nature of our procedure appealing.
We have applied this procedure to two examples: a nonlinear axially restrained elastic beam and a thin plate model obeying Kirchhoff-Love theory \ref{C3}.

We emphasize that a necessary precondition for the hand calculations performed in Sections \ref{sect:beam_mfd} and \ref{sect:plate_computation} is the explicit knowledge of eigenfunctions $\{\phi_n\}_{n \in \mathbb{N}}$ of the elastic operator $A$.
This assumes the physical domain $\Omega$ which the continuum occupies to be simple.
This poses no limitations to beam (or other one-dimensional) examples, but for two- or higher-dimensional examples, $\Omega$ may be severely constrained; e.g., our plate example required a rectangular $\Omega$.
For more complicated domains, numerical techniques have to be employed for the extraction of eigenfunctions.

While this paper was primarily concerned with purely mechanical continuum models, alternative physical phenomena may also fit the setup of Section~\ref{sect:setup}.
For instance, the fluid-conveying pipe model studied in \cite{li2024data} falls into this category (with $\alpha = 1$), under reasonable boundary conditions for which a Poincaré inequality holds for the first derivatives.

\subsection*{Acknowledgements}
We thank Gábor Stépán for suggesting this line of research to us.

\appendix

\section{Identification of the space \texorpdfstring{$\dom(A^\frac12)$}{} for the elastic operator \texorpdfstring{\eqref{eq:elastic_A}}{}}
\label{appendix:domA12}

Disposing of the constant coefficients, the problem reduces to the study of the fourth order differential operator
\begin{displaymath}
    Aw : = \frac{d^4w}{dx^4}
\end{displaymath}
with domain $\dom(A) = H^4([0,l]) \cap \big( \bigcap_{\eta \in F} \ker(\eta) \big)$, with boundary conditions $F \subset H^{-4}([0,l])$ chosen  such that \eqref{eq:SAcond} holds. 

We recall the following result from \cite{russell1988positive}, for the special case of interest here when $A$ is strictly positive.

\begin{theorem}[Theorem 3, \cite{russell1988positive}] \label{thm:P}
    There exists an invertible bounded linear operator $P :L^2([0,l]) \to L^2([0,l])$ such that
    \begin{equation}
        -\frac{\partial^2w}{\partial x^2} = PA^\frac12 w, \qquad w \in \dom(A^\frac12).
        \label{eq:deriv2_vs_A12}
    \end{equation}
\end{theorem}

Explicitly, the map is given on the basis $\{\phi_k\}_{k \in \mathbb{N}}$ (assumed orthonormal in this section) by $P\phi_k  =- \lambda_k^{-\frac12} \partial_{xx} \phi_k = P^{-1} \phi_k$, and if one denotes $\psi_k := P\phi_k$, $k \in \mathbb{N}$, then $\{ \psi_k / \Vert \psi_k \Vert_X \}_{k \in \mathbb{N}}$ form on orthonormal basis of $X$ \cite[page 765]{russell1988positive}.
Note that it is a consequence of Theorem~\ref{thm:P} that $\Vert \psi_k \Vert_X$ is bounded from above and away from $0$ uniformly in $k \in \mathbb{N}$.

An immediate corollary is the following.

\begin{corollary} \label{cor:domA12}
    If $w \in \dom(A^\frac12)$, then there exist $c,C>0$ such that
    \begin{equation}
        c \Vert w \Vert_{A^{\frac12}} \leq \Vert w\Vert_{H^2([0,l])} \leq C \Vert w \Vert_{A^{\frac12}}.
        \label{eq:equiv_norm}
    \end{equation}
    In particular, $\dom(A^\frac12)$ is continuously embedded in $ H^2([0,l])$.
\end{corollary}

\begin{proof}
    The first inequality follows directly from \eqref{eq:deriv2_vs_A12}.
    The second one follows from an application of \eqref{eq:deriv2_vs_A12} once more alongside an interpolation inequality. 
\end{proof}

If $w \in \dom(A^\frac12)$, we moreover have that there exists a sequence of eigenvectors $(\phi_j)_{j \in \mathbb{N}}$ of $A$ such that $\Vert w-\phi_j \Vert_{A^\frac12} \to 0$ as $j \to \infty$.
Hence, by the second inequality in \eqref{eq:equiv_norm}, we have that $w \in \bigcap_{\eta \in G} \ker(\eta)$, where $G := F \cap H^{-2}([0,l])$.

So far we have
\begin{equation}
      \dom(A^\frac12) \subset H^{2}([0,l]) \cap \left( \bigcap_{\eta \in G} \ker(\eta) \right).
      \label{eq:domA12subset}
\end{equation}
If moreover the boundary conditions are such that \eqref{eq:ass_dom12} holds, 
we have equality in \eqref{eq:domA12subset}.
To see this, note that 
\begin{displaymath}
    \dom(A^\frac12)  = \left\{ w \in L^2([0,l]) \; \Big| \; \sum_{k=1}^\infty \lambda_k \langle w , \phi_k \rangle^2_X < \infty \right\}
\end{displaymath}
and
\begin{align*}
    \lambda_k^\frac12 \langle w , \phi_k \rangle_X &= \langle w ,A^\frac12 \phi_k \rangle_X \\
    &= \langle w, P ( - \partial_{xx} \phi_k)\rangle_X \\
    &= -\langle  \partial_{xx}w,P \phi_k \rangle_X \tag{$\dagger$}   \label{eq:IBP} \\ 
    &= -\langle \partial_{xx}w,\psi_k \rangle_X.
\end{align*}
Here $P$ commutes with $\partial_{xx}$ due to its explicit form when applied to $\phi_k$.
The integration by parts in \eqref{eq:IBP} is justified by the following lemma.
Since the $\{\psi_k\}_{k \in \mathbb{N}}$ are mutually orthogonal and span $X$, which are moreover bounded below and above, we obtain the desired conclusion. 

\begin{lemma}
    Suppose \eqref{eq:ass_dom12}. 
    Then, for all $ k \in \mathbb{N}$,
    \begin{displaymath}
        [w \psi'_k - w' \psi_k]_0^l = 0, \qquad \text{for all } w \in H^{2}([0,l]) \cap \left( \bigcap_{\eta \in G} \ker(\eta) \right).
    \end{displaymath}
\end{lemma}

\begin{proof}
    First note that the span of $\{\phi_k\}_{k \in \mathbb{N}}$ is dense in $H^{2}([0,l]) \cap \left( \bigcap_{\eta \in G} \ker(\eta) \right)$ with respect to the $H^2$ norm.
    We may hence write
    \begin{displaymath}
        w = \sum_{j=1}^\infty \langle w , \phi_j \rangle_X \phi_j,
    \end{displaymath}
    with the latter converging in the $H^2$ norm.
    By definition,
    \begin{align*}
        [w \psi'_k - w' \psi_k]_0^l &= - \lambda_k^{-\frac12} [w \phi_k'''-w'\phi_k'']^l_0 \\
        &=- \lambda_k^{-\frac12} \sum_{j=1}^\infty \langle w , \phi_j \rangle_X [\phi_j \phi_k'''-\phi_j'\phi_k'']^l_0.
    \end{align*}
    Since 
    \begin{displaymath}
        \lambda_k \delta_{jk}=\langle \phi_j, A\phi_k \rangle_X = [\phi_j \phi_k'''-\phi_j'\phi_k'']^l_0 + \lambda^\frac12_j \lambda^\frac12_k\langle\psi_j,\psi_k\rangle_X
    \end{displaymath}
    and the $\{\psi_k\}_{k \in \mathbb{N}}$ are pairwise orthogonal, we have that $[\phi_j \phi_k'''-\phi_j'\phi_k'']^l_0 = 0$ whenever $j \neq k$.
    When $j = k$, \eqref{eq:ass_dom12} implies $[\phi_j \phi_j'''-\phi_j'\phi_j'']^l_0=0$, hence the claim of the lemma follows.
\end{proof}

\section{An embedding property of \texorpdfstring{$Y^\beta$}{}}
\label{appendix:domLB}

We prove the following Lemma.

\begin{lemma} \label{lemma:embedding}
    Suppose \ref{A1} with $X = L^2(\Omega)$ for a $\Omega$ either bounded Lipschitz domain in $\mathbb{R}^d$, $\mathbb{R}^d$ itself, or a torus $\Omega = \mathbb{T}^d$, and that elliptic regularity holds for both $A$ and $A^\frac12$, i.e., 
    \begin{subequations} \label{eq:ellipticity_ass}
    \begin{align}
        \Vert u \Vert_{H^{k}(\Omega)} &\lesssim \Vert u \Vert_{A^\frac12}, & & u \in \dom(A^\frac12), \\
        \Vert u \Vert_{H^{2k}(\Omega)} &\lesssim \Vert u \Vert_{A}, & & u \in \dom(A),
    \end{align}   
    \end{subequations}
    for some $k \in \mathbb{N}$.
    Fix $\beta \in (0,1)$.
    Then the fractional power domain $Y^\beta$ associated to $\mathcal{L}_{2\mu A^\frac12}$, defined below Definition~\ref{def:fractpowers} with, is continuously embedded in $H^t(\Omega) \times H^s(\Omega)$ equipped with the norm $\Vert (u,v) \Vert_{H^t \times H^s} =\sup \{ \Vert u \Vert_{H^t},\Vert v \Vert_{H^s} \}$ for $t<k+k\beta$ and $s<k\beta$. 
\end{lemma}

\begin{proof}
    We may increase either $t$ or $s$, or both, so that $s = t-k > 0$, without loss of generality by Sobolev embedding.
    Fix
    \begin{displaymath}
        \theta =  1-\frac{s}{k},
    \end{displaymath}
    so that $\theta \in (1-\beta,1)$.
    By \cite[pages 28-29]{henry1981geometric}, it suffices to show that there exists $C > 0$ such that
    \begin{equation}
        \Vert (u,v) \Vert_{H^t \times H^s} \leq C \Vert \mathcal{L}_{2\mu A^\frac12} (u,v) \Vert_Y^{1-\theta} \Vert (u,v) \Vert_Y^{\theta} \qquad \text{for all } (u,v) \in \dom(\mathcal{L}_{2\mu A^\frac12}).
        \label{eq:AppendixBreq}
    \end{equation}
    We compute, using Sobolev interpolation \cite{brezis2018gagliardo} with $t= k \theta + 2k(1-\theta)$ and $s = k(1-\theta)$,
    \begin{align*}
        \Vert (u,v) \Vert_{H^t \times H^s} &\leq \sup \left\{ \Vert u \Vert_{H^{2k}}^{1-\theta} \Vert u \Vert_{H^k}^\theta ,  \Vert v \Vert_{H^k}^{1-\theta} \Vert v\Vert_{L^2}^{\theta} \right\} \\
        &\leq \Vert (u,v) \Vert^{1-\theta}_{H^{2k} \times H^k} \Vert (u,v) \Vert^{\theta}_{H^k \times L^2}, \qquad \quad (u,v) \in H^{2k}(\Omega) \times H^k(\Omega).
    \end{align*}
    For $(u,v) \in \dom(\mathcal{L}_{2\mu A^\frac12}) = \dom(A) \times \dom(A^\frac12)$ (contained in $H^{2k}(\Omega) \times H^k(\Omega)$ by assumption \eqref{eq:ellipticity_ass}), we moreover have
    \begin{displaymath}
         \Vert (u,v) \Vert_{H^k \times L^2} \lesssim \sup\{\Vert u \Vert_{A^\frac12},\Vert v\Vert_{L^2} \} \leq \Vert (u,v) \Vert_Y
    \end{displaymath}
    and
    \begin{align*}
          \Vert (u,v) \Vert_{H^{2k} \times H^k} &\lesssim  \sup\{\Vert u \Vert_{A},\Vert v\Vert_{A^\frac12} \} \\
          &\lesssim \sup\{\Vert Au \Vert_{L^2},\Vert A^\frac12 v\Vert_{L^2} \}\\
          &\leq \sup\{\Vert Au + 2\mu A^\frac12 v \Vert_{L^2} +  2 \mu\Vert A^\frac12 v\Vert_{L^2} ,\Vert A^\frac12 v\Vert_{L^2} \}\\
          & \lesssim \sup\{\Vert Au + 2\mu A^\frac12 v \Vert_{L^2} ,\Vert A^\frac12 v\Vert_{L^2} \} \\
          &\leq  \Vert \mathcal{L}_{2\mu A^\frac12} (u,v) \Vert_Y  ,
    \end{align*}
    by elliptic regularity \eqref{eq:ellipticity_ass}, upon noting that all graph norms $\Vert \cdot \Vert_{A^\alpha}$ appearing are equivalent to $\Vert A^\alpha \cdot \Vert_{X}$ by strict positivity of $A$.
    This shows \eqref{eq:AppendixBreq} over the desired domain.
\end{proof}

For the case of the beam example, Lemma~\ref{lemma:embedding} is used in Section~\ref{sect:beam_nonlin} with $k =2$, with elliptic regularity furnished by Corollary~\ref{cor:domA12} and the setup of $A$ in Section~\ref{sect:beam_A}.

\section{Bi-orthogonal system for \texorpdfstring{$\mathcal{L}_{2 \mu A^\frac12}$}{} associated to the fourth order differential operator on \texorpdfstring{$L^2([0,1])$}{}}
\label{appendix:biorthogonal}

We record the eigenvectors, adjoint eigenvectors and eigenvalues appearing in Section~\ref{sect:spectrum} for the case $\alpha=\frac12$ and $A = \partial_{xxxx}$ on $L^2([0,1])$, with boundary conditions as in \eqref{eq:BCs_beam} adjusted to $L^2([0,1])$.
The eigenvalues and eigenvectors (normalized according to \eqref{eq:phi_normal}) of $A$ are
\begin{displaymath}
    \sigma_n =  \left( n \pi \right)^4, \qquad \phi_n(x) = \frac{1}{(n\pi)^2}\sin\left( n \pi x \right), \qquad n \in \mathbb{N}.
\end{displaymath}
The eigenvalues of $\mathcal{L}_{2 \mu A^\frac12}$ are
\begin{displaymath}
    \lambda_n^\pm = \left(-\mu \pm i \sqrt{1-\mu^2} \right)  \left( n \pi\right)^2.
\end{displaymath}
The normalized eigenvectors of $\mathcal{L}_{2 \mu A^\frac12}$ are
\begin{displaymath}
    \psi_n^\pm(x) = \begin{pmatrix}
        \frac{1}{(n\pi)^2} \\
        -\mu \pm i \sqrt{1-\mu^2} 
    \end{pmatrix} \sin\left( n \pi x \right),
\end{displaymath}
the adjoint eigenvectors are
\begin{displaymath}
    \psi^{*-}_n(x) =\begin{pmatrix}
        \left( 1 + \frac{i\mu}{\sqrt{1-\mu^2}} \right)\frac{1}{(n\pi)^2} \\
        \frac{i} {\sqrt{1-\mu^2}}
    \end{pmatrix}  \sin\left( n \pi x \right), \qquad \psi^{*+}_n = \overline{\psi^{*-}_n}.
\end{displaymath}

\section{Bi-orthogonal system for \texorpdfstring{$\mathcal{L}_{2 \mu A^\frac12}$}{} associated to the fourth order differential operator on \texorpdfstring{$L^2([0,a] \times [0,b])$}{}}
\label{appendix:biorthogonal_plate}

We record the eigenvectors, adjoint eigenvectors and eigenvalues appearing in Section~\ref{sect:spectrum} for the case $\alpha=\frac12$ and $A = \Delta^2$ on $L^2([0,1])$, with boundary conditions as in \eqref{eq:plate_BCs}, i.e., the setup of Section~\ref{sect:plate}.
The eigenvalues and eigenvectors (normalized according to \eqref{eq:phi_normal}) of $A$ are
\begin{displaymath}
      \sigma_{nm} = \left[ \left(\frac{n\pi}{a}\right)^2 + \left(\frac{m\pi}{b}\right)^2\right]^2,
\qquad \phi_{nm}(x) = \sqrt{\frac{2}{ab \sigma_{nm}}} \sin\left(\frac{n\pi x}{a}\right)
\sin\left(\frac{m\pi y}{b}\right), 
\end{displaymath}
$n,m \in \mathbb{N}$.
The eigenvalues of $\mathcal{L}_{2 \mu A^\frac12}$ are
\begin{displaymath}
    \lambda_{nm}^\pm = \left(-\mu \pm i \sqrt{1-\mu^2} \right)  \left[ \left(\frac{n\pi}{a}\right)^2 + \left(\frac{m\pi}{b}\right)^2\right].
\end{displaymath}
The normalized eigenvectors of $\mathcal{L}_{2 \mu A^\frac12}$ are
\begin{displaymath}
    \psi_{nm}^\pm(x) = \begin{pmatrix}
       \sigma_{nm}^{-\frac12}  \\
        -\mu \pm i \sqrt{1-\mu^2} 
    \end{pmatrix} \sqrt{\frac{2}{ab }} \sin\left(\frac{n\pi x}{a}\right)
\sin\left(\frac{m\pi y}{b}\right),
\end{displaymath}
the adjoint eigenvectors are
\begin{displaymath}
    \psi^{*-}_{pq}(x) =\begin{pmatrix}
        \left( 1 + \frac{i\mu}{\sqrt{1-\mu^2}} \right)\sigma_{pq}^{-\frac12} \\
        \frac{i} {\sqrt{1-\mu^2}}
    \end{pmatrix}  \sqrt{\frac{2}{ab }} \sin\left(\frac{p \pi x}{a}\right)
\sin\left(\frac{q\pi y}{b}\right), \qquad \psi^{*+}_{pq} = \overline{\psi^{*-}_{pq}}.
\end{displaymath}

\printbibliography

\end{document}